%% file: systolic_paper.tex
\documentclass[12pt]{article}

\usepackage{amssymb}
\usepackage{latexsym}
\usepackage{amsmath}
\usepackage{amsthm}
\usepackage{amscd}
\usepackage[dvips]{graphics}
\usepackage{overpic}
\usepackage[mathscr]{euscript}
\usepackage{url}

\usepackage[usenames,dvipsnames]{color}
\usepackage{hyperref}

\usepackage{mdframed}
\usepackage{verbatim}

\usepackage{color}

\newcommand{\mc}[1]{{\mathcal #1}}

\numberwithin{equation}{section}

\newtheorem{theorem}{Theorem}[section]
\newtheorem{proposition}[theorem]{Proposition}
\newtheorem{corollary}[theorem]{Corollary}
\newtheorem{lemma}[theorem]{Lemma}
\newtheorem{lemma-definition}[theorem]{Lemma-Definition}

\newtheorem{conjecture}[theorem]{Conjecture}

\newtheorem{question}[theorem]{Question}

\theoremstyle{definition}
\newtheorem{definition}[theorem]{Definition}
\newtheorem{remark}[theorem]{Remark}

\newcommand{\floor}[1]{\left\lfloor #1 \right\rfloor}
\newcommand{\ceil}[1]{\left\lceil #1 \right\rceil}

\newcommand{\C}{{\mathbb C}}

\newcommand{\R}{{\mathbb R}}

\newcommand{\Z}{{\mathbb Z}}

\newcommand{\op}{\operatorname}

\newcommand{\Ker}{\op{Ker}}

\newcommand{\Tr}{\op{Tr}}

\newcommand{\tensor}{\otimes}

\newcommand{\vu}{\nu}

\newcommand{\bpm}{\begin{pmatrix}}
\newcommand{\epm}{\end{pmatrix}}

\renewcommand{\epsilon}{\varepsilon}

\newcommand{\Sg}{\text{Sp}(2)}
\newcommand{\Sgt}{\widetilde{\text{Sp}}(2)}
\newcommand{\Df}{\text{Diff}(S^1)}
\newcommand{\Dft}{\widetilde{\text{Diff}}(S^1)}

\begin{document}

\setcounter{tocdepth}{2}

\title{Computing Reeb dynamics on 4d convex polytopes}
\author{Julian Chaidez\footnote{Partially supported by an NSF Graduate Research Fellowship.}\;  and Michael Hutchings\footnote{Partially supported by NSF grant DMS-1708899, a Simons Fellowship, and a Humboldt Research Award.}}

\maketitle

\begin{abstract} 
We study the combinatorial Reeb flow on the boundary of a four-dimensional convex polytope. We establish a correspondence between ``combinatorial Reeb orbits'' for a polytope, and ordinary Reeb orbits for a smoothing of the polytope, respecting action and Conley-Zehnder index. One can then use a computer to find all combinatorial Reeb orbits up to a given action and Conley-Zehnder index. We present some results of experiments testing Viterbo's conjecture and related conjectures. In particular, we have found some new examples of polytopes with systolic ratio $1$.
\end{abstract}

\input{s1_introduction_and_main_results.tex}

\input{s2_type_1_reeb_orbits.tex}

\input{s3_reeb_dynamics_on_polytopes.tex}

\input{s4_quaternionic.tex}

\input{s5_dynamics_on_smoothings.tex}

\input{s6_smooth_combinatorial.tex}

\begin{appendix}

\input{a1_rotation_numbers.tex}

\end{appendix}

\end{document}

%% file: s1_introduction_and_main_results.tex
\section{Introduction And main results}
\label{sec:introduction_and_main_results}

This paper is about computational methods for testing Viterbo's conjecture and related conjectures, via combinatorial Reeb dynamics.

\subsection{Review of Viterbo's conjecture}
\label{sec:reviewviterbo}

We first recall two different versions of Viterbo's conjecture.
Consider $\R^{2n}=\C^n$ with coordinates $z_i=x_i+\sqrt{-1}y_i$ for $i=1,\ldots,n$. Define the standard Liouville form
\[
\lambda_0 = \frac{1}{2}\sum_{i=1}^n\left(x_i\,dy_i - y_i\,dx_i\right).
\]
Let $X$ be a compact domain in $\R^{2n}$ with smooth boundary $Y$. Assume that $X$ is ``star-shaped'', by which we mean that $Y$ is transverse to the radial vector field. Then the $1$-form $\lambda = \lambda_0|_Y$ is a contact form on $Y$. Associated to $\lambda$ are the contact structure $\xi=\Ker(\lambda)\subset TY$ and the Reeb vector field $R$ on $Y$, characterized by $d\lambda(R,\cdot)=0$ and $\lambda(R)=1$. A {\bf Reeb orbit\/} is a periodic orbit of $R$, i.e. a map $\gamma:\R/T\Z\to Y$ for some $T>0$ such that $\gamma'(t)=R(\gamma(t))$, modulo reparametrization. The {\bf symplectic action\/} of a Reeb orbit $\gamma$, denoted by $\mc{A}(\gamma)$, is the period of $\gamma$, or equivalently
\begin{equation}
\label{eqn:symplecticaction}
\mc{A}(\gamma) = \int_{\R/T\Z}\gamma^*\lambda_0.
\end{equation}

Reeb orbits on $Y$ always exist. This was first proved by Rabinowitz \cite{rabinowitz} and is a special case of the Weinstein conjecture; see \cite{tw} for a survey. We are interested here in the minimal period of a Reeb orbit on $Y$, which we denote by $\mc{A}_{\op{min}}(X)\in(0,\infty)$, and its relation to the volume $\op{vol}(X)$ of $X$ with respect to the Lebesgue measure. For this purpose, define the {\bf systolic ratio\/}
\[
\op{sys}(X) = \frac{\mc{A}_{\op{min}}(X)^n}{n!\op{vol}(X)}.
\]
The exponent ensures that the systolic ratio of $X$ is invariant under scaling of $X$; and the constant factor is chosen so that if $X$ is a ball then $\op{sys}(X)=1$.

\begin{conjecture}[weak Viterbo conjecture]
\label{conj:vweak}
Let $X\subset\R^{2n}$ be a compact convex domain with smooth boundary such that $0\in\op{int}(X)$. Then $\op{sys}(X)\le 1$.
\end{conjecture}

Conjecture~\ref{conj:vweak} asserts that among compact convex domains with the same volume, $\mc{A}_{\op{min}}$ is largest for a ball. Although the role of the convexity hypothesis is somewhat mysterious, some hypothesis beyond the star-shaped condition is necessary: it is shown in \cite{abhs} that there exist star-shaped domains in $\R^4$ with arbitrarily large systolic ratio\footnote{It is further shown in \cite{abhs2} that there are star-shaped domains in $\R^4$ which are {\em dynamically convex\/} (meaning that every Reeb orbit on the boundary has rotation number greater than $1$, see Proposition~\ref{prop:ehwz}(a) below) and have systolic ratio $2-\epsilon$ for $\epsilon>0$ arbitrarily small.}.
One motivation for studying Conjecture~\ref{conj:vweak} is that it implies the Mahler conjecture in convex geometry \cite{ako}. 

To put Conjecture~\ref{conj:vweak} in more context, recall\footnote{The precise definition of ``symplectic capacity'' varies in the literature. For an older but extensive survey of symplectic capacities see \cite{chls}.} that a {\bf symplectic capacity\/} is a function $c$ mapping some class of $2n$-dimensional symplectic manifolds to $[0,\infty]$, such that:
\begin{itemize}
\item (Monotonicity)
If there exists a symplectic embedding $\varphi:(X,\omega)\to(X',\omega')$, then $c(X,\omega)\le c(X',\omega')$.
\item (Conformality)
If $r>0$ then $c(X,r\omega)=rc(X,\omega)$.
\end{itemize}
Of course we can regard (open) domains in $\R^{2n}$ as symplectic manifolds with the restriction of the standard symplectic form $\omega=\sum_{i=1}^ndx_i\,dy_i$. Conformality for a domain $X\subset \R^{2n}$ means that $c(rX)=r^2c(X)$. 

Following the usual convention in symplectic geometry, for $r>0$ define the ball
\[
B(r)=\left\{z\in\C^n\;\big|\; \pi|z|^2\le r\right\}
\]
and the cylinder
\[
Z(r)=\left\{z\in\C^n\;\big|\; \pi|z_1|^2\le r\right\}.
\]
We say that a symplectic capacity $c$ is {\bf normalized\/} if it is defined at least for all compact convex domains in $\R^{2n}$ and if
\[
c(B(r))=c(Z(r))=r.
\]
Note that the symplectic capacity $c(Z(r))$ is defined as the limit of $c(E_i)$, where $E_i \subset \R^{2n}$ is a sequence of ellipsoids exausting $Z(r)$.

An example of a normalized symplectic capacity is the {\bf Gromov width\/} $c_{\op{Gr}}$, where $c_{\op{Gr}}(X,\omega)$ is defined to be the supremum over $r$ such that there exists a symplectic embedding $B(r)\to (X,\omega)$. It is immediate from the definition that $c_{\op{Gr}}$ is monotone and conformal. Since symplectomorphisms preserve volume, we have $c_{\op{Gr}}(B(r))=r$; and the Gromov nonsqueezing theorem asserts that $c_{\op{Gr}}(Z(r))=r$.

Another example of a normalized symplectic capacity is the {\bf Ekeland-Hofer-Zehnder capacity\/}, denoted by $c_{\op{EHZ}}$. If $X$ is a compact convex domain with smooth boundary such that $0\in\op{int}(X)$, then\footnote{Since translations act by symplectomorphism on $\R^{2n}$, the symplectic capacities of $X$ are invariant under translation. However, we will often assume that $0\in\op{int}(X)$ so that we can sensibly discuss the Reeb flow on $\partial X$.}
\begin{equation}
\label{eqn:cehzamin}
c_{\op{EHZ}}(X) = \mc{A}_{\op{min}}(X).
\end{equation}
This is explained in \cite[Thm.\ 2.2]{artsteinostrover2014}, combining results from \cite{eh1,hz}.

Any symplectic capacity which is defined for compact convex domains in $\R^{2n}$ with smooth boundary is a $C^0$ continuous function of the domain (i.e., continuous with respect to the Hausdorff distance between compact sets), and thus extends uniquely to a $C^0$ continuous function of all compact convex sets in $\R^{2n}$.

\begin{conjecture}
[strong Viterbo conjecture\footnote{The original version of Viterbo's conjecture from \cite{viterbo} asserts that a normalized symplectic capacity, restricted to convex sets in $\R^{2n}$ of a given volume, takes its maximum on a ball. (This follows from what we are calling the ``strong Viterbo conjecture'' and implies what we are calling the ``weak Viterbo conjecture''.) Viterbo further conjectured that the maximum is achieved only if the interior of the convex set is symplectomorphic to an open ball; cf.\ Question~\ref{question:zoll_ball} below.}]
\label{conj:vstrong}
All normalized symplectic capacities agree on compact convex sets in $\R^{2n}$.
\end{conjecture}

\begin{remark} Convexity is a key hypothesis in both the weak and strong versions of the Viterbo conjecture. For star-shaped domains that are not convex, counterexamples to the conclusion of the strong Viterbo conjecture were given in \cite[Thm.\ 1.12]{hermann}, and counterexamples to the conclusion of the weak Viterbo conjecture were given later in \cite[Thm.\ 2]{abhs}. In \cite[Cor.\ 5.2]{ghr}, it is shown exactly where the conclusions of the strong and original Viterbo conjectures start to fail in a certain family of non-convex examples.
\end{remark}

Conjecture~\ref{conj:vstrong} implies Conjecture~\ref{conj:vweak}, because if Conjecture~\ref{conj:vstrong} holds, and if $X$ is a compact convex domain with smooth boundary and $0\in\op{int}(X)$, then
\[
\mc{A}_{\op{min}}(X)^n = c_{\op{EHZ}}(X)^n = c_{\op{Gr}}(X)^n \le n!\op{vol}(X).
\]
Here the second equality holds by Conjecture~\ref{conj:vstrong}; and the inequality on the right holds because if there exists a symplectic embedding $B(r)\to X$, then $r^n/n! = \op{vol}(B(r)) \le \op{vol}(X)$.

There are also interesting families of non-normalized symplectic capacities.
For example, there are the Ekeland-Hofer capacities defined in \cite{eh}; more recently, and conjecturally equivalently, positive $S^1$-equivariant symplectic homology was used in \cite{gh} to define a symplectic capacity $c_k^{S^1}$ for each integer $k\ge 1$. Each equivariant capacity $c_k^{S^1}(X)$ is the symplectic action of some Reeb orbit, which when $X$ is generic (so that $\lambda$ is nondegenerate) has Conley-Zehnder index $n-1+2k$ (see \S\ref{sec:rotcz} below). Some other symplectic capacities give the total action of a finite set of Reeb orbits, such as the ECH capacities in the four-dimensional case \cite{qech}, or the symplectic capacities defined by Siegel using rational symplectic field theory \cite{siegel}.

Conjectures~\ref{conj:vweak} and \ref{conj:vstrong} are known for some special examples such as $S^1$-invariant convex domains \cite{ghr}, but they have not been well tested more generally. To test Conjecture~\ref{conj:vweak}, and as a first step towards computing other symplectic capacities and testing conjectures about them, we need good methods for computing Reeb orbits, their actions, and their Conley-Zehnder indices. The plan in this paper is to understand Reeb orbits on a smooth convex domain in terms of ``combinatorial Reeb orbits'' on convex polytopes approximating the domain.

\subsection{Combinatorial Reeb orbits}
\label{sec:cro}

Let $X$ be any compact convex set in $\R^{2n}$ with $0\in\op{int}(X)$, and let $y\in\partial X$. The {\bf tangent cone\/}, which we denote by $T_y^+X$, is the closure of the set of vectors $v$ such $y+\epsilon v\in X$ for some $\epsilon>0$. For example, if $\partial X$ is smooth at $y$, then $T_y^+X$ is a closed half-space whose boundary is the usual tangent space $T_y\partial X$.

Also define the {\bf positive normal cone\/}
\[
N_y^+X = \left\{v\in\R^{2n}\;\big|\;\langle x-y,v\rangle\le 0\;\;\forall x\in X\right\}.
\]
If $\partial X$ is smooth at $y$, then $N_y^+X$ is a one-dimensional ray and consists of the outward pointing normal vectors to $\partial X$ at $y$. 

Finally, define the {\bf Reeb cone\/}
\[
R_y^+X = T_y^+X\cap {\mathbf i}N_y^+X
\]
where ${\mathbf i}$ denotes the standard complex structure on $\C^n=\R^{2n}$. We show that $R^+_yX$ is nonempty in the cases of interest for this paper in Lemma \ref{lem:wp1}. If $\partial X$ is smooth near $y$, then $R_y^+X$ is the ray consisting of nonnegative multiples of the Reeb vector field on $\partial X$ at $y$. Indeed, in this case we can write
\[
T_y\partial X = \left\{v\in\R^{2n}\;\big|\;\langle \nu,v\rangle=0\right\}
\]
where $\nu$ is the outward unit normal vector to $\partial X$ at $y$; and the Reeb vector field at $y$ is given by
\begin{equation}
\label{eqn:Reebinu}
R_y = 2\frac{{\mathbf i}\nu}{\langle \nu,y\rangle}.
\end{equation}

\begin{figure}[h!]
\label{fig:cones_on_polytopes}
\begin{center}
\includegraphics[width=.6\linewidth]{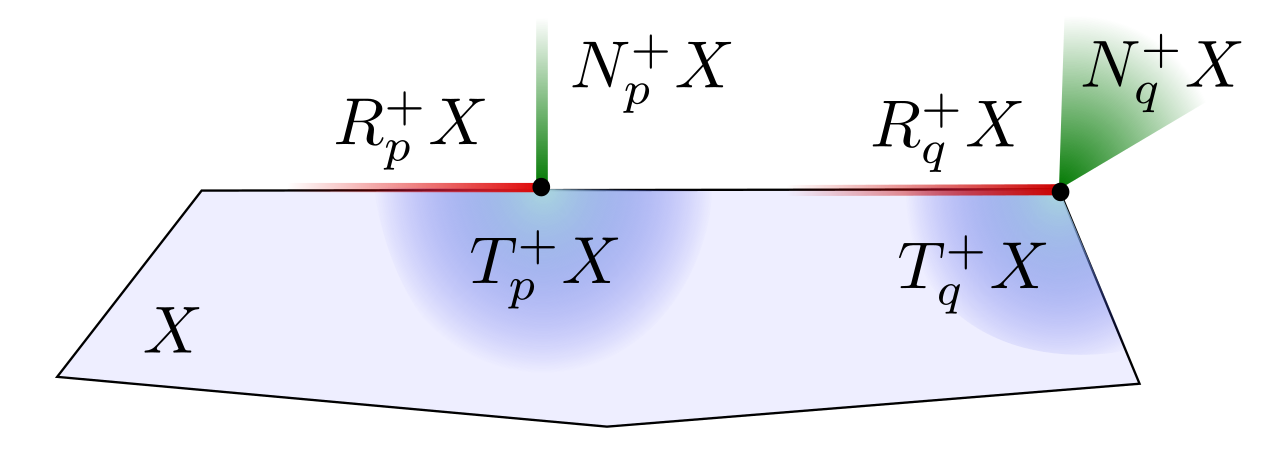}
\end{center}
\caption{We depict the tangent, normal and Reeb cones for two points $p,q \in X$ in a polytope $X \subset \R^2$.}
\end{figure}

Suppose now that $X$ is a convex polytope (i.e.\ a compact set given by the intersection of a finite set of closed half-spaces) in $\R^{2n}$ with $0\in\op{int}(X)$.
Our convention is that a {\bf $k$-face\/} of $X$ is a $k$-dimensional subset $F\subset \partial X$ which is the interior of the intersection with $\partial X$ of some set of the hyperplanes defining $X$. For a given $k$-face $F$, the tangent cone $T_y^+X$, the positive normal cone $N_y^+X$, and the Reeb cone $R_y^+X$ are the same for all $y\in F$. Thus we can denote these cones by $T_F^+X$, $N_F^+X$, and $R_F^+X$ respectively.

We will usually restrict attention to polytopes of the following type:

\begin{definition}
\label{def:symplectic_polytope}
A {\bf symplectic polytope\/} in $\R^4$ is a convex polytope $X$ in $\R^4$ such that $0\in\op{int}(X)$ and no $2$-face of $X$ is Lagrangian, i.e., the standard symplectic form $\omega_0 = \sum_{i=1}^2dx_i\,dy_i$ restricts to a nonzero $2$-form on each $2$-face.
\end{definition}

Symplectic polytopes are generic, in the sense that in the space of polytopes in $\R^4$ with a given number of $3$-faces, the set of non-symplectic polytopes is a proper subvariety. Moreover, the boundary of a symplectic polytope in $\mathbb{R}^4$ has a well-posed ``combinatorial Reeb flow'' in the following sense\footnote{There is also a more general notion of ``generalized Reeb trajectory'' on the boundary of a compact convex convex set in $\R^{2n}$ whose interior contains the origin; see Definition~\ref{def:gro} below. We do not know whether the generalized Reeb flow on the boundary of a four-dimensional symplectic polytope is well posed.}.

 \begin{proposition}[Lemma \ref{lem:wp1}]
 \label{prop:well-posed}
 If $X$ is a symplectic polytope in $\R^4$, then the Reeb cone $R_F^+X$ is one-dimensional for each face $F$.
\end{proposition}

\begin{definition}
\label{def:cro}
Let $X$ be a symplectic polytope in $\R^4$. A {\bf combinatorial Reeb orbit\/} for $X$ is a finite sequence $\gamma=(\Gamma_1,\ldots,\Gamma_k)$ of oriented line segments in $\partial X$, modulo cyclic permutations, such that for each $i=1,\ldots,k$:
\begin{itemize}
\item The final endpoint of $\Gamma_i$ agrees with the initial endpoint of $\Gamma_{i+1\mod k}$.
\item
There is a face $F$ of $X$ such that $\op{int}(\Gamma_i)\subset F$, the endpoints of $\Gamma_i$ are on the boundary of (the closure of) $F$, and $\Gamma_i$ points in the direction of $R_F^+X$.
\end{itemize}
The {\bf combinatorial symplectic action\/} of a combinatorial Reeb orbit as above is defined by
\[
\mc{A}_{\op{comb}}(\gamma)=\sum_{i=1}^k\int_{\Gamma_i}\lambda_0.
\]
\end{definition}

To give a better idea of what combinatorial Reeb orbits look like, we have the following lemma. 

\begin{lemma}
\label{lem:Reebcone}
(proved in \S\ref{sec:drc})
Let $X$ be a symplectic polytope in $\R^4$. Then the Reeb cones of the faces of $X$ satisfy the following:
\begin{itemize}
\item
If $E$ is a 3-face, then $R_E^+X$ consists of all nonnegative multiples of the Reeb vector field on $E$.
\item If $F$ is a $2$-face, then $R_F^+X$ points into a 3-face $E$ adjacent to $F$, and agrees with $R_E^+X$. 
\item If $L$ is a $1$-face, then one of the following possibilities holds:
\begin{itemize}
\item $R_L^+X$ points into a $3$-face $E$ adjacent to $L$ and agrees with $R_E^+X$. In this case we say that $L$ is a {\bf good\/} $1$-face.
\item $R_L^+X$ is tangent to $L$, and does not agree with $R_E^+X$ for any of the $3$-faces $E$ adjacent to $L$. In this case we say that $L$ is a {\bf bad\/} $1$-face.
\end{itemize}
\item If $P$ is a $0$-face, then $R_P^+X$ points into a $3$-face $E$ or bad $1$-face $L$ adjacent to $F$ and agrees with $R_E^+X$ or $R_L^+X$ respectively.
\end{itemize}
\end{lemma}

\begin{remark}
The reason we assume that $X$ has no Lagrangian $2$-faces in Definition~\ref{def:symplectic_polytope} is that if $F$ is a Lagrangian 2-face, then $R_F^+X$ is two-dimensional and tangent to $F$. In fact, $\partial R_F^+X = R_{E_1}^+X\cup R_{E_2}^+X$ where $E_1$ and $E_2$ are the two $3$-faces adjacent to $F$. In this case we do not have a well-posed ``combinatorial Reeb flow'' on $\partial X$.
\end{remark}

\begin{definition}
A combinatorial Reeb orbit as above is:
\begin{itemize}
\item {\bf Type 1\/} if it does not intersect the $1$-skeleton of $X$;
\item {\bf Type 2\/} if it intersects the $1$-skeleton of $X$, but only in finitely many points which are some of the endpoints of the line segments $\Gamma_i$;
\item {\bf Type 3\/} if it contains a bad $1$-face.
\end{itemize}
\end{definition}

\begin{figure}[h!]
\label{fig:types_of_orbits}
\includegraphics[width=\linewidth]{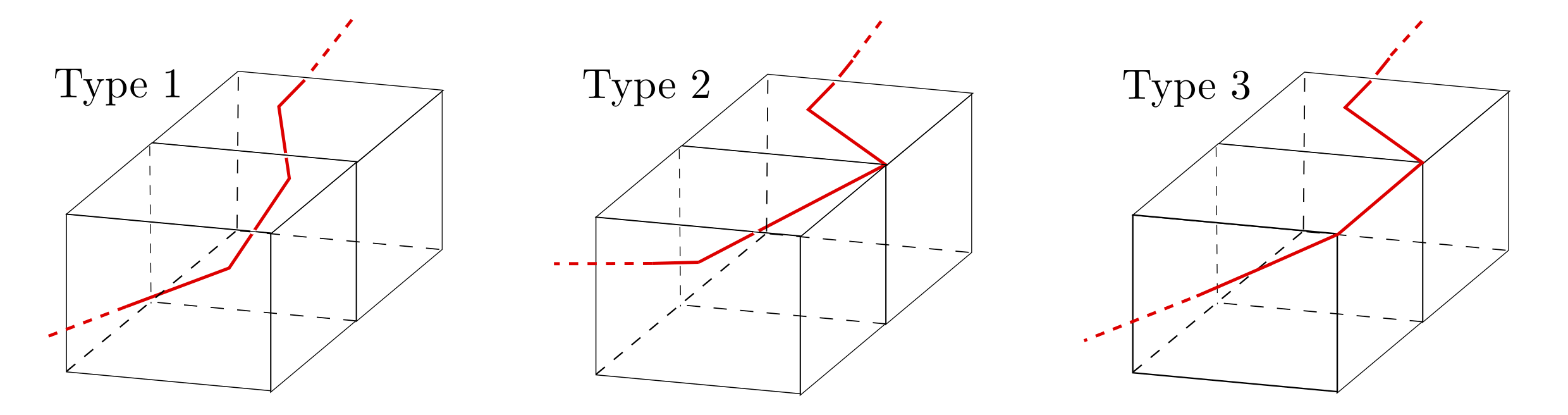}
\caption{We depict sub-trajectories of the three types of orbits, in red. Each cube above represents a 3-face of a hypothetical 4-polytope.}
\end{figure}

It follows from the definitions that each combinatorial Reeb orbit is of one of the above three types. Type 1 Reeb orbits are the most important for our computations. We expect that Type 2 combinatorial Reeb orbits do not exist for generic polytopes; see Conjecture~\ref{conj:genericity} below. Type 3 combinatorial Reeb orbits generally cannot be eliminated by perturbing the polytope; but we will see in Theorem~\ref{thm:smoothtocomb}(iii) below that they do not contribute to the symplectic capacities that we are interested in. See Remark~\ref{rem:spiral} for some intuition for this.

\subsection{Rotation numbers and the Conley-Zehnder index}
\label{sec:rotcz}

Let $X$ be a compact star-shaped domain in $\R^4$ with smooth boundary $Y$. Let $\Phi_t:Y\to Y$ denote the time $t$ flow of the Reeb vector field $R$. The derivative of $\Phi_t$ preserves the contact form $\lambda$ and so defines a map on the contact structure $\xi = \Ker(\lambda)$, namely
\[
d\Phi_t:\xi_y \longrightarrow \xi_{\Phi_t(y)}
\]
for each $y\in Y$. The map $d\Phi_t$ is symplectic with respect to the symplectic form $d\lambda|_\xi$ on $\xi$.

We say that a Reeb orbit $\gamma:\R/T\Z\to Y$ is {\bf nondegenerate\/} if the ``linearized return map''
\begin{equation}
\label{eqn:dPhiT}
d\Phi_T:\xi_{\gamma(0)}\longrightarrow \xi_{\gamma(0)}
\end{equation}
does not have $1$ as an eigenvalue. The contact form $\lambda$ is called nondegenerate if all Reeb orbits are nondegenerate.

Now fix a symplectic trivialization $\tau:\xi\to Y\times\R^2$. If $\gamma$ is a Reeb orbit as above, then the trivialization $\tau$ allows us to regard the map \eqref{eqn:dPhiT} as an element of the $2$-dimensional symplectic group $\op{Sp}(2)$. Moreover, the family of maps
\begin{equation}
\label{eqn:fom}
\left\{
\R^2 \stackrel{\tau^{-1}}{\longrightarrow} \xi_{\gamma(0)} \stackrel{d\Phi_t}{\longrightarrow} \xi_{\gamma(t)} \stackrel{\tau}{\longrightarrow} \R^2\right\}_{t\in[0,T]}
\end{equation}
defines a path $\phi_\tau$ in $\Sg$ from the identity to the map \eqref{eqn:dPhiT}, and thus an element of the universal cover $\Sgt$ of $\Sg$. As we review in Appendix~\ref{app:rotation_numbers}, any element of $\Sgt$ has a well-defined {\bf rotation number\/}. We denote the rotation number of $\phi_\tau$ by
\[
\rho(\gamma)\in\R.
\]

Note that the rotation number $\rho(\gamma)$ does not depend on the choice of symplectic trivialization $\tau$ of $\xi$. Since $Y\simeq S^3$, any two such trivializations are homotopic, giving rise to a homotopy of paths \eqref{eqn:fom} whose final endpoints are conjugate in $\op{Sp}(2)$. Invariance of the rotation number then follows from Lemma~\ref{lem:compute_rho_bar}.

If $\gamma$ is nondegenerate (which holds automatically when $\rho(\gamma)$ is not an integer), then the {\bf Conley-Zehnder index\/} of $\gamma$ is defined by
\begin{equation}
\label{eqn:CZrot}
\op{CZ}(\gamma) = \floor{\rho(\gamma)} + \ceil{\rho(\gamma)} \in \Z.
\end{equation}

\begin{proposition}
\label{prop:ehwz}
Let $X$ be a compact strictly convex domain in $\R^4$ with smooth boundary $Y$ and with $0\in\op{int}(X)$. Then:
\begin{itemize}
\item[\emph{(a)}] 
Every Reeb orbit $\gamma$ in $Y$ has $\rho(\gamma)>1$. In particular, if $\gamma$ is nondegenerate then $\op{CZ}(\gamma)\ge 3$.
\item[\emph{(b)}] 
There exists a Reeb orbit $\gamma$ which is action minimizing, i.e.\ $\mc{A}(\gamma) = \mc{A}_{\op{min}}(X)$, with
\[
\rho(\gamma) \le 2.
\]
If $\gamma$ is also nondegenerate then the inequality is strict, so that $\op{CZ}(\gamma)=3$.
\end{itemize}
\end{proposition}

\begin{proof}
(a) was proved by Hofer-Wysocki-Zehnder \cite{hwz}.

(b) follows from the construction of the Ekeland-Hofer-Zehnder capacity and an index calculation of Hu-Long \cite{hulong2002}. In fact, it was recently shown by Abbondandolo-Kang \cite{ak} and Irie \cite{irie} that $c_{\op{EHZ}}(X)$ agrees with a capacity defined from symplectic homology, which by construction is the action of some Reeb orbit $\gamma$ with $\rho(\gamma)\le 2$, with equality only if $\gamma$ is degenerate.
\end{proof}

Suppose now that $X$ is a symplectic polytope in $\R^4$. As we explain in Definition~\ref{def:crn}, each Type 1 combinatorial Reeb orbit $\gamma$ has a well-defined {\bf combinatorial rotation number\/}, which we denote by $\rho_{\op{comb}}(\gamma)\in\R$. There is also a combinatorial notion of nondegeneracy for $\gamma$, which automatically holds when $\rho_{\op{comb}}(\gamma)\notin\Z$. When $\gamma$ is a nondegenerate Type 1 combinatorial Reeb orbit, we can then define its {\bf combinatorial Conley-Zehnder index\/} by analogy with \eqref{eqn:CZrot} as
\begin{equation}
\label{eqn:ccz}
\op{CZ}_{\op{comb}}(\gamma) = \floor{\rho_{\op{comb}}(\gamma)} + \ceil{\rho_{\op{comb}}(\gamma)}.
\end{equation}
The combinatorial rotation number and combinatorial Conley-Zehnder index of a Type 2 combinatorial Reeb orbit are not defined; and although we do not need this, it would be natural to define the combinatorial rotation number and combinatorial Conley-Zehnder index of a Type 3 combinatorial Reeb orbit to be $+\infty$.

\subsection{Smooth-combinatorial correspondence}

Let $X$ be a convex polytope in $\R^{2n}$. If $\epsilon>0$, define the {\bf $\epsilon$-smoothing\/} of $X$ by
\begin{equation}
\label{eqn:deltasmoothing}
X_\epsilon = \left\{z\in\R^{2n} \;\big|\; \op{dist}(z,X)\le \epsilon \right\}.
\end{equation}
The domain $X_\epsilon$ is convex and has $C^1$-smooth boundary. The boundary is $C^\infty$ smooth except along strata arising from the boundaries of the faces of $X$; see \S\ref{sec:smoothings} for a detailed description.

Our main results are the following two theorems, giving a correspondence between combinatorial Reeb dynamics on a symplectic polytope in $\R^4$, and ordinary Reeb dynamics on $\epsilon$-smoothings of the polytope.

There is a slight technical issue here: since $\partial X_\epsilon$ is only $C^1$ smooth, the Reeb vector field on $\partial X_\epsilon$ is only $C^0$, so that for a Reeb orbit $\gamma$, the linearized Reeb flow \eqref{eqn:dPhiT} might not be defined. If $\gamma$ is transverse to the strata where $\partial X_\epsilon$ is not $C^\infty$ (which is presumably true for all $\gamma$ if $X$ and $\epsilon$ are generic), then the Reeb flow in a neighborhood of $\gamma$ has a well-defined linearization; we call such orbits {\bf linearizable\/}. It turns out that a non-linearizable Reeb orbit $\gamma$ on $\partial X_\epsilon$ still has a well-defined rotation number $\rho(\gamma)$, defined in \S\ref{sec:srn}.

The following theorem describes how combinatorial Reeb orbits give rise to Reeb orbits on smoothings. See Lemma~\ref{lem:combtosmooth} for a more precise statement.

\begin{theorem}
\label{thm:combtosmooth}
(proved in \S\ref{sec:combtosmooth})
Let $X$ be a symplectic polytope in $\R^4$, and let $\gamma$ be a nondegenerate Type 1 combinatorial Reeb orbit for $X$. Then for all $\epsilon>0$ sufficiently small, there is a distinguished Reeb orbit $\gamma_\epsilon$ on $\partial X_\epsilon$ such that:
\begin{itemize}
	\item[\emph{(i)}] $\gamma_\epsilon$ converges in $C^0$ to $\gamma$ as $\epsilon\to0$.
	\item[\emph{(ii)}] $\lim_{\epsilon\to 0}\mc{A}(\gamma_\epsilon) = \mc{A}_{\op{comb}}(\gamma)$.
	\item[\emph{(iii)}] $\gamma_\epsilon$ is linearizable and nondegenerate, $\rho(\gamma_\epsilon) = \rho_{\op{comb}}(\gamma)$, and $\op{CZ}(\gamma_\epsilon) = \op{CZ}_{\op{comb}}(\gamma)$.
\end{itemize}
\end{theorem}

The following theorem describes how Reeb orbits on smoothings give rise to combinatorial Reeb orbits.

\begin{theorem}
\label{thm:smoothtocomb}
(proved in \S\ref{sec:smoothtocomb})
Let $X$ be a symplectic polytope in $\R^4$. Then there are constants $c_F>0$ for each $0$-, $1$-, or $2$-face $F$ of $X$ with the following property.

Let $\{(\epsilon_i,\gamma_i)\}_{i=1,\ldots}$ be a sequence of pairs such that $\epsilon_i>0$; $\gamma_i$ is a Reeb orbit on $\partial X_{\epsilon_i}$; and $\epsilon_i\to 0$ as $i\to\infty$. Suppose that $\rho(\gamma_i)<R$ where $R$ does not depend on $i$. Then after passing to a subsequence, there is a combinatorial Reeb orbit $\gamma$ for $X$ such that:
\begin{itemize}
	\item[\emph{(i)}] $\gamma_i$ converges in $C^0$ to $\gamma$ as $i\to\infty$.
    \item[\emph{(i)}] $\lim_{i\to\infty}\mc{A}(\gamma_i) = \mc{A}_{\op{comb}}(\gamma)$.
    \item[\emph{(iii)}] $\gamma$ is either Type 1 or Type 2.
    \item[\emph{(iv)}] If $\gamma$ is Type 1, then for $i$ sufficiently large, $\gamma_i$ is linearizable and $\rho(\gamma_i) = \rho_{\op{comb}}(\gamma)$. If $\gamma$ is also nondegenerate, then for $i$ sufficiently large, $\gamma_i$ is nondegenerate and $\op{CZ}(\gamma_i) = \op{CZ}_{\op{comb}}(\gamma)$.
    \item[\emph{(v)}] Let $F_1,\ldots,F_k$ denote the faces containing the endpoints of the segments of the combinatorial Reeb orbit $\gamma$. Then
\begin{equation}
\label{eqn:segmentbound}
\sum_{i=1}^kc_{F_i}\le R.
\end{equation}
\end{itemize}
\end{theorem}

\begin{remark}
One can compute explicit constants $c_F$ -- see \S\ref{sec:smoothtocomb} for the details -- and the resulting bound \eqref{eqn:segmentbound} is crucial in enabling finite computations. For example, combinatorial Reeb orbits with a given action bound could have arbitrarily many segments winding in a ``helix'' around a bad $1$-face. However the bound \eqref{eqn:segmentbound} ensures that combinatorial Reeb orbits with too many segments will not arise as limits of sequences of smooth Reeb orbits with bounded rotation number.
\end{remark}

\begin{remark}
The methods of this paper can be used to prove a version of Theorem \ref{thm:combtosmooth} (omitting the condition (c) on the rotation number and Conley-Zehnder index) for polytopes $X \subset \R^{2n}$ for $2n > 4$, under the hypothesis that the $(2n-2)$-faces of $X$ are symplectic. Generalizing Theorem~\ref{thm:smoothtocomb} to higher dimensions would be less straightforward, as its proof in four dimensions depends crucially on estimates on the rotation number in \S\ref{sec:smoothingdynamics}. Higher dimensional analogues of these estimates are an interesting topic for future work.
\end{remark}

Theorem~\ref{thm:smoothtocomb} allows one to compute the EHZ capacity of a four-dimensional polytope as follows:

\begin{corollary}
\label{cor:computecehz}
Let $X$ be a symplectic polytope in $\R^4$. Then
\begin{equation}
\label{eqn:corcehz}
c_{\op{EHZ}}(X) = \op{min}\{\mc{A}_{\op{comb}}(\gamma)\}
\end{equation}
where the minimum is over combinatorial Reeb orbits $\gamma$ with $\sum_ic_{F_i}\le 2$ which are either Type 1 with $\rho_{\op{comb}}(\gamma)\le 2$ or Type 2.
\end{corollary}

\begin{remark}
If the coordinates of the vertices of $X$ are rational, then the combinatorial action of every combinatorial Reeb orbit is rational. It follows from Theorem~\ref{thm:smoothtocomb} that in this case, $c_{\op{EHZ}}(X)$, as well as the other symplectic capacities mentioned in \S\ref{sec:reviewviterbo} determined by actions of Reeb orbits, are all rational.
\end{remark}

To explain why Corollary~\ref{cor:computecehz} follows from Theorem~\ref{thm:smoothtocomb}, we need to recall a result of K\"unzle \cite{kunzle} as explained by Artstein-Avidan and Ostrover \cite{artsteinostrover2014}. 

\begin{definition}
\label{def:gro}
If $X$ is any compact convex set in $\R^{2n}$ with $0\in\op{int}(X)$, a {\bf generalized Reeb orbit\/} for $X$ is a map $\gamma:\R/T\Z\to\partial X$ for some $T>0$ such that $\gamma$ is continuous and has left and right derivatives at every point, which agree for almost every $t$, and the left and right derivatives at $t$ are in $R_{\gamma(t)}^+X$. If $\gamma$ is a generalized Reeb orbit, define its symplectic action by \eqref{eqn:symplecticaction}.
\end{definition}

\begin{proposition}
\cite[Prop.\ 2.7]{artsteinostrover2014}
\label{prop:aao}
If $X$ is a compact convex set in $\R^{2n}$ with $0\in\op{int}(X)$, then
\[
c_{\op{EHZ}}(X) = \op{min}\{\mc{A}(\gamma)\}
\]
where the minimum is taken over all generalized Reeb orbits.
\end{proposition}

\begin{proof}[Proof of Corollary~\ref{cor:computecehz}.]
Pick a sequence of positive numbers $\epsilon_i$ with $\lim_{i\to\infty} \epsilon_i = 0$. For each $i$, by equation \eqref{eqn:cehzamin}, we can find a Reeb orbit $\gamma_i$ on $\partial X_{\epsilon_i}$ with $\mc{A}(\gamma_i) = c_{\op{EHZ}}(X_{\epsilon_i})$. By Proposition~\ref{prop:ehwz}(b), we can assume that $\rho(\gamma_i)\le 2$. By Theorem~\ref{thm:smoothtocomb}, it follows that after passing to a subsequence, there is a combinatorial Reeb orbit $\gamma$ for $X$, satisying the conditions in Corollary~\ref{cor:computecehz}, such that
\[
\mc{A}_{\op{comb}}(\gamma) = \lim_{i\to\infty}\mc{A}(\gamma_i) = \lim_{k\to\infty}c_{\op{EHZ}}(X_{\epsilon_i}) = c_{\op{EHZ}}(X).
\]
Here the last equality holds by the $C^0$ continuity of $c_{\op{EHZ}}$. We conclude that
\[
c_{\op{EHZ}}(X)\ge \op{min}\{\mc{A}_{\op{comb}}(\gamma)\}
\]
where the minimum is over combinatorial Reeb orbits $\gamma$ satisfying the conditions in Corollary~\ref{cor:computecehz}.

The reverse inequality follows from Proposition~\ref{prop:aao}, because by Definitions~\ref{def:cro} and \ref{def:gro}, every combinatorial Reeb orbit is a generalized Reeb orbit. (For a symplectic polytope in $\R^4$, a ``generalized Reeb orbit'' is equivalent to a generalization of a ``combinatorial Reeb orbit'' in which there may be infinitely many line segments.)
\end{proof}

\begin{remark}
Haim-Kislev \cite[Thm.\ 1.1]{haim-kislev} gives a different formula for $c_{\op{EHZ}}$ of a convex polytope, which is valid in $\R^{2n}$ for all $n$. 
That formula implies that in the minimum \eqref{eqn:corcehz}, we can also assume that $\gamma$ has at most one segment in each $3$-face.
\end{remark}

\subsection{Experiments testing Viterbo's conjecture}

If $X$ is a convex polytope in $\R^{2n}$, define its systolic ratio by
\[
\op{sys}(X) = \frac{c_{\op{EHZ}}(X)^n}{n!\op{vol}(X)}.
\]
Note that $c_{\op{EHZ}}$ is translation invariant, so we can make this definition without assuming that $0\in\op{int}(X)$.

Since every compact convex domain in $\R^{2n}$ can be $C^0$ approximated by convex polytopes, it follows that the weak version of Viterbo's conjecture, namely Conjecture~\ref{conj:vweak}, is true if and only if every convex polytope $X$ has systolic ratio $\op{sys}(X)\le 1$. The combinatorial formula for the systolic ratio given by Corollary~\ref{cor:computecehz} allows us to test this conjecture by computer when $n=2$. In particular, we ran optimization algorithms over the space of $k$-vertex convex polytopes in $\R^4$ to find local maxima of the systolic ratio\footnote{This is a somewhat involved process; convergence to a local maximum becomes very slow once one is close. It helps to mod out the space of polytopes by the $15$-dimensional symmetry group generated by translations, linear symplectomorphisms, and scaling. To find exact local maxima, one can look at symplectic invariants, such as areas of $2$-faces, and guess what these are converging to.}. In the results below, when listing the vertices of specific polytopes, we use Lagrangian coordinates $(x_1,x_2,y_1,y_2)$.

\subsection*{5-vertex polytopes (4-simplices).} Experimentally\footnote{Perhaps this could be proved analytically using the formula in \cite[Thm.\ 1.1]{haim-kislev}.}, every $4$-simplex $X$ has systolic ratio
\[
\op{sys}(X) \le 3/4.
\]
The apparent maximum of $3/4$ is achieved by the ``standard simplex'' with vertices
\[
(0,0,0,0), (1,0,0,0), (0,1,0,0), (0,0,1,0), (0,0,0,1).
\]

\begin{remark}
Corollary~\ref{cor:computecehz} does not directly apply to (a translate of) this polytope because it has some Lagrangian $2$-faces. For examples like these, we find numerically that a slight perturbation of the polytope to a symplectic polytope (to which Corollary~\ref{cor:computecehz} does apply) has systolic ratio very close to the claimed value. One can compute the systolic ratio of a polytope with Lagrangian $2$-faces rigorously using a generalization of Corollary~\ref{cor:computecehz}. For the particular example above, one can also compute the systolic ratio by hand using \cite[Thm.\ 1.1]{haim-kislev}.
\end{remark}

We have found families of other examples of 4-simplices with systolic ratio $3/4$, including some with no Lagrangian $2$-faces. An example is the simplex with vertices
\[
(0,0,0,0), (1,-1/3,0,0), (0,-1/3,1,0), (-2/3,-1,2/3,0), (0,0,0,1).
\]

\subsection*{6-vertex polytopes.} We found families of 6-vertex polytopes with systolic ratio equal to $1$. An example is the polytope with vertices
\[
(0,0,0,0), (1,0,0,0), (0,0,1,0), (0,0,0,1), (0,-1,1,0), (-1,-1,0,1).
\]
(Apparently the previous minimum number of vertices of a known example with systolic ratio $1$ was 12, given by the Lagrangian product of a triangle and a square \cite[Lem.\ 5.3.1]{schlenk}. Some more examples of Lagrangian products with systolic ratio 1 are presented in \cite{balitskiy}.)

\subsection*{7-vertex polytopes.} We also found families of $7$-vertex polytopes with systolic ratio $1$. One example has vertices
\begin{gather*}
(0,0,0,0), (1,0,0,0), (0,0,1,0), (0,0,0,1),\\
 (1/3,-2/3,2/3,0), (-1,-1,0,1/2), (0,0,1/3,-1/3).
\end{gather*}
Presumably there exist $k$-vertex polytopes in $\R^4$ with systolic ratio equal to $1$ for every $k\ge 6$.

\subsection*{The 24-cell.} We also found a special example of a polytope with systolic ratio $1$: a rotation of the 24-cell (one of the six regular polytopes in four dimensions). See \S\ref{sec:24_cell} for details.

We have heavily searched the spaces of polytopes with $7$ or fewer vertices and have not found any counterexamples to Viterbo's conjecture. For polytopes with $8$ vertices, our computer program starts becoming slower (taking seconds to minutes per polytope on a standard laptop), and we have not yet searched as extensively.

\subsection*{Towards a proof of the weak Viterbo conjecture?}
Let $X$ be a star-shaped domain in $\R^4$ with smooth boundary $Y$. Following \cite{abhs}, we say that $X$ is {\bf Zoll\/} if every point on $Y$ is contained in a Reeb orbit with minimal action. Note that:
\begin{itemize}
	\item[\emph{(a)}] If $X$ is strictly convex and a local maximizer for the systolic ratio of convex domains in the $C^0$ topology, then $X$ is Zoll.
	\item[\emph{(b)}] If $X$ is Zoll, then $X$ has systolic ratio $\op{sys}(X)=1$.
\end{itemize}
Part (a) holds because if $X$ is strictly convex and if $y\in Y$ is not on an action mimizing Reeb orbit, then one can shave some volume off of $X$ near $y$ without creating any new Reeb orbits of small action. Part (b) holds by a topological argument going back to \cite{weinstein}. (In fact one can further show that $X$ is symplectomorphic to a closed ball; see \cite[Prop.\ 4.3]{abhs}.) Of course, these observations are not enough to prove Conjecture~\ref{conj:vweak}, since we do not know that the systolic ratio for convex domains takes a maximum, let alone on a strictly convex domain. But this does suggest the following strategy for proving Conjecture~\ref{conj:vweak} via convex polytopes.

\begin{definition}
\label{def:combinatorially_Zoll}
Let $X$ be a convex polytope in $\R^4$ with $0\in\op{int}(X)$. We say that $X$ is {\bf combinatorially Zoll\/} if there is an open dense subset $U$ of $\partial X$ such that every point in $U$ is contained in a combinatorial Reeb orbit (avoiding any Lagrangian $2$-faces of $X$) with combinatorial action equal to $c_{\op{EHZ}}(X)$.
\end{definition}

We have checked by hand that the above examples of polytopes with systolic ratio equal to $1$ are combinatorially Zoll. This suggests:

\begin{conjecture}
Let $X$ be a convex polytope in $\R^4$ with $0\in\op{int}(X)$. Then:
\begin{itemize}
\item[\emph{(a)}] If $X$ is combinatorially Zoll, then $\op{sys}(X)=1$.
\item[\emph{(b)}] If $k$ is sufficiently large ($k\ge 6$ might suffice) and if $X$ maximizes systolic ratio over convex polytopes with $\le k$ vertices, then $X$ is combinatorially Zoll.
\end{itemize}
\end{conjecture}

Part (a) of this conjecture can probably be proved following the argument in the smooth case. Part (b) might be much harder. But both parts of the conjecture together would imply the weak Viterbo conjecture (using a compactness argument to show that for each $k$ the systolic ratio takes a maximum on the space of convex polytopes with $\le k$ vertices).

\begin{question}
\label{question:zoll_ball}
If a convex polytope $X$ in $\R^4$ is combinatorially Zoll, then is $\op{int}(X)$ symplectomorphic to an open ball?
\end{question}

\subsection{Experiments testing other conjectures}
\label{sec:otherexp}

One can also use Theorems~\ref{thm:combtosmooth} and \ref{thm:smoothtocomb} to test conjectures about Reeb orbits that do not have minimal action. For example, if $X$ is a convex domain with smooth boundary and $0\in\op{int}(X)$ such that ${\lambda_0}|_{\partial X}$ is nondegenerate, and if $k$ is a positive integer, define
\begin{equation}
\label{eqn:defak}
\mc{A}_k(X) = \op{min}\{\mc{A}(\gamma)\mid\op{CZ}(\gamma) = 2k+1\},
\end{equation}
where the minimum is over Reeb orbits $\gamma$ on $\partial X$. In particular $\mc{A}_1(X) = \mc{A}_{\op{min}}(X)$ by Proposition~\ref{prop:ehwz}(b).

\begin{conjecture}
\label{conj:A2}
For $X$ as above we have $\mc{A}_2(X) \le 2\mc{A}_1(X)$.
\end{conjecture}

This conjecture has nontrivial content when every action-minimizing Reeb orbit has rotation number at least $3/2$. (If an action-minimizing Reeb orbit has rotation number less than $3/2$, then its double cover has Conley-Zehnder index $5$ and thus verifies the conjectured inequality.) To explain how to test this, we need the following definitions.

\begin{definition}
Let $X$ be a symplectic polytope in $\R^4$. Let $L>0$. We say that $X$ is {\bf $L$-nondegenerate\/} if:
\begin{itemize}
	\item $X$ does not have any Type 2 combinatorial Reeb orbit $\gamma$ with $\mc{A}_{\op{comb}}(\gamma)\le L$.
	\item Every Type 1 combinatorial Reeb orbit $\gamma$ with $\mc{A}_{\op{comb}}(\gamma)\le L$ is nondegenerate, see Definition~\ref{def:crn}.
\end{itemize}
\end{definition}

It follows from Theorem~\ref{thm:smoothtocomb} that if a symplectic polytope $X$ is $L$-nondegenerate, then for all $\epsilon>0$ sufficiently small, all Reeb orbits on $\partial X_\epsilon$ with action less than $L$ are nondegenerate.

\begin{conjecture}
\label{conj:genericity}
For any integer $k$ and any real number $L$, the set of $L$-nondegenerate symplectic polytopes with $k$ vertices is dense in the set of all $k$-vertex convex polytopes containing $0$, topologized as an open subset of $\R^{4k}$.
\end{conjecture}

\begin{definition}
Let $k$ be a positive integer and let $X$ be a symplectic polytope in $\R^4$. Suppose that $X$ is $L$-nondegenerate and has a combinatorial Reeb orbit $\gamma$ with $\mc{A}(\gamma)<L$ and $\op{CZ}_{\op{comb}}(\gamma)=2k+1$. By analogy with \eqref{eqn:defak}, define
\[
\mc{A}_k^{\op{comb}}(X)=\op{min}\left\{\mc{A}_{\op{comb}}(\gamma) \mid \op{CZ}_{\op{comb}}(\gamma) = 2k+1\right\}
\]
where the minimum is over combinatorial Reeb orbits $\gamma$ with combinatorial action less than $L$.
\end{definition}

Conjecture~\ref{conj:A2} is now equivalent\footnote{More precisely, by Theorem~\ref{thm:combtosmooth}, if $X$ is a polytope as above for which $\mc{A}_1^{\op{comb}}(X)$ and $\mc{A}_2^{\op{comb}}(X)$ are defined, and if
$\mc{A}_2^{\op{comb}}(X) > 2\mc{A}_1^{\op{comb}}(X)$,
then Conjecture~\ref{conj:A2} fails for (nondegenerate $C^\infty$ perturbations of) $\epsilon$-smoothings of $X$ for $\epsilon$ sufficiently small. Thus Conjecture~\ref{conj:A2} implies Conjecture~\ref{conj:A2comb}. If Conjecture~\ref{conj:genericity} is true, then one can conversely show, by approximating smooth domains by $L$-nondegenerate symplectic polytopes, that Conjecture~\ref{conj:A2comb} implies Conjecture~\ref{conj:A2}.} to the following:

\begin{conjecture}
\label{conj:A2comb}
Let $X$ be a symplectic polytope in $\R^4$. Assume that $\mc{A}_1^{\op{comb}}(X)$ and $\mc{A}_2^{\op{comb}}(X)$ are defined. Then
\[
\mc{A}_2^{\op{comb}}(X) \le 2\mc{A}_1^{\op{comb}}(X).
\]
\end{conjecture}

One can use Theorems~\ref{thm:combtosmooth} and \ref{thm:smoothtocomb} to compute $\mc{A}_k^{\op{comb}}(X)$. One can then test Conjecture~\ref{conj:A2comb} by using optimization algorithms to try to maximize the ratio $\mc{A}_2^{\op{comb}}(X)/(2\mc{A}_1^{\op{comb}}(X))$. So far we have not found any example where this ratio is greater than $1$.

\subsection*{The rest of the paper}

In \S\ref{sec:type1}, we investigate Type 1 combinatorial Reeb orbits in detail, we define the combinatorial rotation number, and we work out the example of the 24-cell. In \S\ref{sec:rdsp}, we establish foundational facts about the combinatorial Reeb flow on a symplectic polytope. In \S\ref{sec:quaternionic} we review a symplectic trivialization of the contact structure on a star-shaped hypersurface in $\R^4$ defined using the quaternions. We explain a key curvature identity due to Hryniewicz and Salom\~ao which implies that in the convex case, the rotation number of a Reeb trajectory increases monotonically as it evolves. In \S\ref{sec:smoothingdynamics} we study the Reeb flow on a smoothing of a polytope. In \S\ref{sec:correspondence} we use this work to prove the smooth-combinatorial correspondence of Theorems~\ref{thm:combtosmooth} and \ref{thm:smoothtocomb}. In the appendix, we review basic facts about rotation numbers that we need throughout.

\subsection*{Acknowledgments.} We thank A.\ Abbondandolo, P.\ Haim-Kislev, U.\ Hryniewicz, and Y.\ Ostrover for helpful conversations, and A. Balitskiy for pointing out some additional references. JC was partially supported by an NSF Graduate Research Fellowship. MH was partially supported by NSF grant DMS-1708899, a Simons Fellowship, and a Humboldt Research Award.

%% file: s2_type_1_reeb_orbits.tex
\section{Type 1 combinatorial Reeb orbits}
\label{sec:type1}

Let $X$ be a symplectic polytope in $\R^4$. In this section we give what amounts to an algorithm for finding the Type 1 combinatorial Reeb orbits and their combinatorial symplectic actions, see Proposition~\ref{prop:orbitbijection}. (Our actual computer implementation uses various optimizations not discussed here.) We also define combinatorial rotation numbers and work out the example of the 24-cell.

\subsection{Symplectic flow graphs}
\label{subsec:symplectic_flow_graphs}

We start by defining ``symplectic flow graphs'' in any even dimension. In the next subsection (\S \ref{subsubsec:symplectic_polytopes}), we will specialize to certain $2$-dimensional flow graphs that keep track of the combinatorics needed to find Type 1 Reeb orbits on the boundary of a symplectic polytope in $\R^4$. 

\begin{definition}
\label{def:linear_domain}
A {\bf linear domain\/} is an intersection of a finite number of open or closed half-spaces in an affine space, or an affine space itself.
\end{definition}

\begin{definition}
\label{def:tangent_space} The {\bf tangent space} $TA$ of a linear domain $A$ is the tangent space $T_xA$ for any $x\in A$; the tangent spaces for different $x$ are canonically isomorphic to each other via translations.
\end{definition}

\begin{definition}
\label{def:affine_map_of_linear_domains}
Let $A$ and $B$ be linear domains. An {\bf affine map} $\phi:A \to B$ is the restriction of an affine map between affine spaces containing $A$ and $B$. Such a map induces a map on tangent spaces which we denote by $T\phi: TA\to TB$.
\end{definition}

\begin{definition}
\label{def:linear_flow}
Let $A$ and $B$ be linear domains. A {\bf linear flow} from $A$ to $B$ is a triple $\Phi = (D,\phi,f)$ consisting of:
\begin{itemize}
	\item the {\bf domain of definition}: a linear domain $D \subset A$.
	\item the {\bf flow map}: an affine map $\phi:D \to B$.
	\item the {\bf action function}: an affine function $f:D \to \R$.
\end{itemize}
We sometimes write $\Phi:A\to B$. In the examples of interest for us, $\phi$ is injective, and $f\ge 0$.
\end{definition}

\begin{definition}
\label{def:linear_flow_composition}
Let $\Phi = (D,\phi,f)$ be a linear flow from $A$ to $B$ and let $\Psi = (E,\psi,g)$ be a linear flow from $B$ to $C$. Their {\bf composition} is the linear flow $\Psi \circ \Phi: A \to C$ defined by
\[
\Psi \circ \Phi = (\phi^{-1}(E),\psi \circ \phi, f + g \circ \phi).
\]
\end{definition}

\begin{remark} Composition of linear flows is associative, and there is an identity linear flow $\iota_A:A \to A$ given by $\iota_A = (A,\op{id}_A,0)$. If $\Phi_i=(D_i,\phi_i,f_i)$ is a linear flow from $A_{i-1}$ to $A_i$ for $i=1,\ldots,k$, and if $\Phi=(D,\phi,f)$ is the composition $\Phi_k\circ\cdots\circ \Phi_1$, then for $x\in D$, we have
\begin{equation}
\label{eqn:actioncomposition}
f(x) = \sum_{i=1}^kf_i((\phi_{i-1}\circ\cdots\circ\phi_1)(x)).
\end{equation}
\end{remark}

\begin{definition}
\label{def:linear_flow_graph}
A {\bf linear flow graph} $G$ is a triple $G = (\Gamma,A,\Phi)$ consisting of:
\begin{itemize}
 \item A directed graph $\Gamma$ with vertex set $V(\Gamma)$ and edge set $E(\Gamma)$.
 \item For each vertex $v$ of $\Gamma$, an open linear domain $A_v$.
 \item For each edge $e$ of $\Gamma$ from $u$ to $v$, a linear flow $\Phi_e = (D_e,\phi_e,f_e):A_u \to A_v$.
\end{itemize}
\end{definition}

\begin{figure}[h!]
\label{fig:flow_graph}
\includegraphics[width=\linewidth]{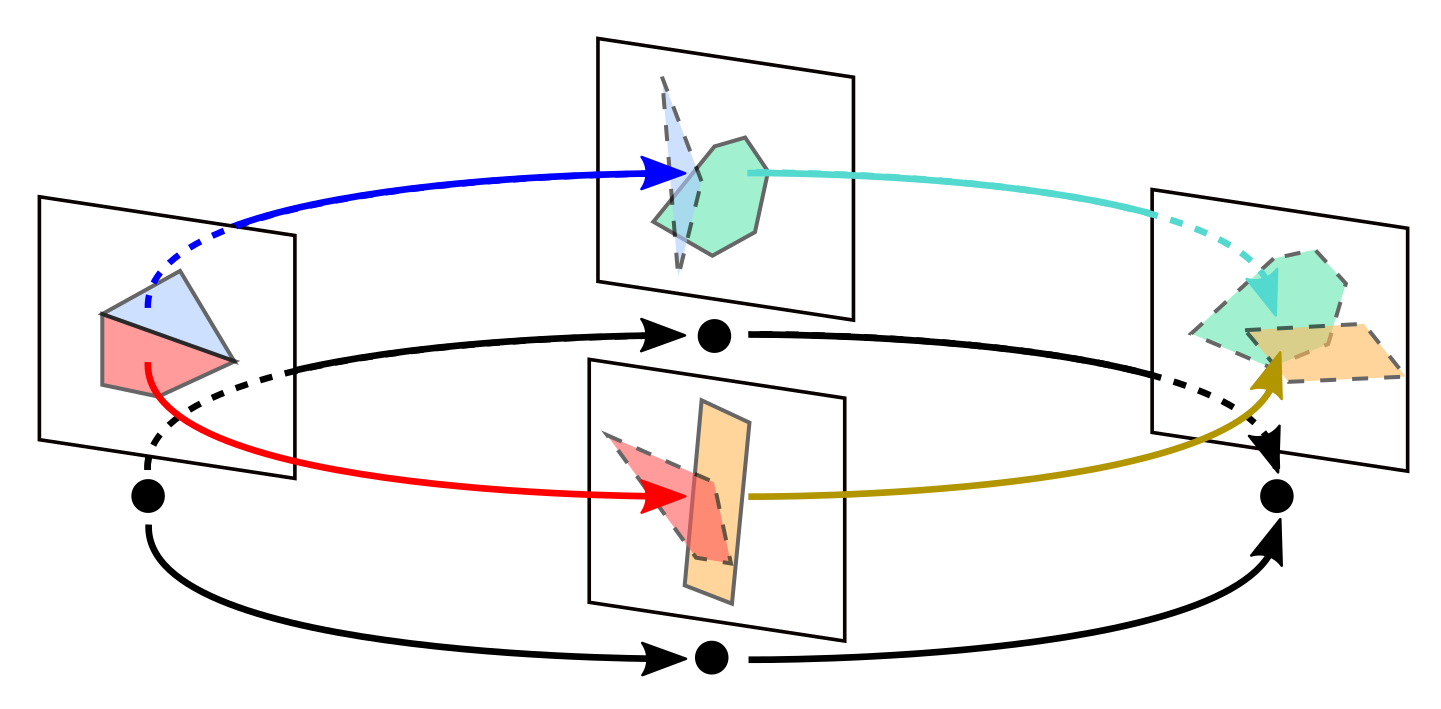}
\caption{An example of a flow graph with 4 nodes and 4 edges. The linear domains and flows are depicted above their corresponding nodes and edges.}
\end{figure}

Let $G=(\Gamma,A,\Phi)$ be a linear flow graph. If $p = e_1\dots e_k$ is a path in $\Gamma$ from $u$ to $v$, we define an associated linear flow
\[
\Phi_p = (D_p,\phi_p,f_p) : A_u \longrightarrow A_v
\]
by
\[
\Phi_p = \Phi_{e_k} \circ \dots \circ \Phi_{e_1}.
\]

\begin{definition}
\label{def:flow_graph_trajectory}
A {\bf trajectory} $\gamma$ of $G$ is a pair $\gamma = (p,x)$, where $p$ is a path in $\Gamma$ and $x \in D_p$.
\end{definition}

\begin{definition}
\label{def:flow_graph_periodic_orbit}
A {\bf periodic orbit} of $G$ is an equivalence class of trajectories $\gamma = (p,x)$ where $p$ is a cycle in $\Gamma$ and $x$ is a fixed point of $\phi_p$, i.e.\ $\phi_p(x) = x$. Two such trajectories $\gamma = (p,x)$ and $\eta = (q,y)$ are equivalent if there are paths $r$ and $s$ in $\Gamma$ such that $p = rs$, $q = sr$, and $\phi_r(x) = y$. We often abuse notation and denote the periodic orbit by $\gamma=(p,x)$, instead of by the equivalence class thereof.
\end{definition}

\begin{definition}
\label{def:action_of_periodic_orbit}
The {\bf action} of a periodic orbit $\gamma = (p,x)$ is defined by $f(\gamma) = f_p(x)$.
\end{definition}

\begin{definition} \label{def:degenerate_orbit} A periodic orbit $\gamma = (p,x)$, where $p$ is a cycle based at $u$, is {\bf degenerate} if the induced map on tangent spaces $T\phi_p:TD_u \to TD_u$ has $1$ as an eigenvalue. Otherwise we say that $\gamma$ is {\bf nondegenerate\/}.
\end{definition}

\begin{definition}
\label{def:symplectic_flow_graph}
An $2n$-dimensional {\bf symplectic flow graph\/} $G$ is a quadruple $G = (\Gamma,A,\omega,\Phi)$ where:
\begin{itemize}
	\item $(\Gamma,A,\Phi)$ is a linear flow graph in which each linear domain $A_v$ has dimension $2n$.
	\item $\omega$ assigns to each vertex $v$ of $\Gamma$ a linear symplectic form $\omega_v$ on $TA_v$.
  \end{itemize}
We require that if $e$ is an edge from $u$ to $v$, then $\phi_e^*\omega_v = \omega_u$.
\end{definition}

\subsection{The symplectic flow graph of a 4d symplectic polytope}
\label{subsubsec:symplectic_polytopes}

\begin{definition}
\label{def:sfgp}
Let $X$ be a symplectic polytope in $\R^4$. We associate to $X$ the two-dimensional symplectic flow graph $G(X)=(\Gamma,A,\omega,\Phi)$ defined as follows:
\begin{itemize}
    \item The vertex set of $\Gamma$ is the set of $2$-faces of $X$. The linear domain associated to a vertex is simply the corresponding $2$-face, regarded as a linear domain in $\R^4$. If $F$ is a $2$-face, then the symplectic form $\omega_F$ on $TF$ is the restriction of the standard symplectic form $\omega_0$ on $\R^4$.
    \item If $F_1$ and $F_2$ are $2$-faces, then there is an edge $e$ in $\Gamma$ from $F_1$ to $F_2$ if and only if there is a $3$-face $E$ adjacent to $F_1$ and $F_2$, and a trajectory of the Reeb vector field $R_E$ on $E$ from some point in $F_1$ to some point in $F_2$. In this case, the linear flow
    \[
\Phi_e = (D_e,\phi_e,f_e):F_1\longrightarrow F_2
    \]
    is defined as follows:
\begin{itemize}
  \item
  The domain $D_e$ is the set of $x\in F_1$ such that there exists a trajectory of $R_E$ from $x$ to some point $y\in F_2$.
  \item
  For $x$ as above, $\phi_e(x)=y$, and $f_e(x)$ is the time it takes to flow along the vector field $R_E$ from $x$ to $y$, or equivalently the integral of $\lambda_0$ along the line segment from $x$ to $y$.
 \end{itemize}
\end{itemize} 
\end{definition}

In the above definition, note that $\phi_e$ and $f_e$ are affine, because the vector field $R_E$ on $E$ is constant by equation \eqref{eqn:Reebinu}.  A simple calculation as in \cite[Eq.\ (5.10)]{hwz} shows that the map $\phi_e$ is symplectic.

\begin{proposition}
\label{prop:orbitbijection}
Let $X$ be a symplectic polytope in $\R^4$. Then there is a canonical bijection
\[
\{\mbox{periodic orbits of $G(X)$}\} \longleftrightarrow \{\mbox{Type $1$ combinatorial Reeb orbits of $X$}\}.
\]
If $(p,x)$ is a periodic orbit of $G(X)$, and if $\gamma$ is the corresponding combinatorial Reeb orbit, then
\begin{equation}
\label{eqn:identifyactions}
f(p,x) = \mc{A}_{\op{comb}}(\gamma).
\end{equation}
\end{proposition}

\begin{proof}
Suppose $(p=e_1\cdots e_k,x)$ is a periodic orbit of $G(X)$. Let $E_i$ denote the $3$-face of $X$ associated to $e_i$. There is then a combinatorial Reeb orbit $\gamma=(L_1,\ldots,L_k)$, where $L_i$ is the line segment in $E_i$ from $\phi_{e-1}\circ\cdots\circ \phi_{e_1}(x)$ to $\phi_{e_i}\circ\cdots\circ \phi_{e_1}(x)$. It follows from Definitions~\ref{def:cro} and \ref{def:sfgp} that this construction defines a bijection from periodic orbits of $G(X)$ to combinatorial Reeb orbits of $X$. The identification of actions \eqref{eqn:identifyactions} follows from equation \eqref{eqn:actioncomposition}.
\end{proof}

By Proposition~\ref{prop:orbitbijection}, to find the Type 1 Reeb orbits\footnote{When testing Viterbo's conjecture and related conjectures,
although all Type 1 orbits of $X$ are detected by the flow graph $G(X)$, in view of Corollary 1.13 we must also account for Type 2 orbits. One can do this by either (1) extending $G(X)$ to a flow graph that includes the lower-dimensional faces of $X$ or (2) working with a flow graph $G(X)$ whose linear domains $A_F$  are the closures of the $2$-faces, rather than $2$-faces themselves. We use the first strategy in our computer program.} of $X$, one can compute the symplectic flow graph $G(X)=(\Gamma,A,\omega,\Phi)$, enumerate the cycles in the graph $\Gamma$, and for each cycle $p$, compute the fixed points of the map $\phi_p$ in the domain $D_p$. In order to avoid searching for arbitrarily long cycles in the graph $\Gamma$ in the cases of interest, we now need to discuss combinatorial rotation numbers.

\subsection{Combinatorial rotation numbers}

\begin{definition}
\label{def:sfg_trivialization}
A {\bf trivialization} of a $2n$-dimensional symplectic flow graph $G=(\Gamma,A,\omega,\Phi)$ is a pair $(\tau,\widetilde{\phi})$ consisting of:
\begin{itemize}
  \item For each vertex $u$ of $\Gamma$, an isomorphism of symplectic vector spaces
  \[
  \tau_u:(TA_u,\omega_u) \stackrel{\simeq}{\longrightarrow} (\R^{2n},\omega_0).
  \]
  \item For each edge $e$ in $\Gamma$ from $u$ to $v$, a lift $\widetilde{\phi}_{e,\tau} \in \widetilde{\op{Sp}}(2n)$ of the symplectic matrix
  \[
  \tau_v \circ T\phi_e \circ \tau_u^{-1}\in\op{Sp}(2n).
  \]
\end{itemize}
Here $\omega_0$ denotes the standard symplectic form on $\R^{2n}$, and $\widetilde{\op{Sp}}(2n)$ denotes the universal cover of the symplectic group $\op{Sp}(2n)$. We sometimes abuse notation and denote the trivialization $(\tau,\widetilde{\phi})$ simply by $\tau$.
\end{definition}

If $p = e_1\dots e_n$ is a path in $\Gamma$ from $u$ to $v$, we define
\[
\widetilde{\phi}_{p,\tau} = \widetilde{\phi}_{e_n,\tau} \circ\cdots\circ \widetilde{\phi}_{e_1,\tau} \in \widetilde{\op{Sp}}(2n).
\]

\begin{definition}
Let $G=(\Gamma,A,\omega,\Phi)$ be a $2$-dimensional symplectic flow graph, let $\tau$ be a trivialization of $G$, and let $p$ be a path in $\Gamma$. Define the {\bf rotation number\/} of $p$ with respect to $\tau$ by
\[
\rho_\tau(p) = \rho(\widetilde{\phi}_{p,\tau})\in\R,
\]
where the right hand side is the rotation number on $\widetilde{\op{Sp}}(2)$ reviewed in Appendix~\ref{app:rotation_numbers}. 
\end{definition}

Suppose now that $X$ is a symplectic polytope in $\R^4$. We now define a canonical trivialization $\tau$ of the symplectic flow graph $G(X)$ which has the useful property that if $(p,x)$ is a periodic orbit of $G(X)$, and if $\gamma$ is the corresponding combinatorial Reeb orbit on $X$ from Proposition~\ref{prop:orbitbijection}, then the rotation number $\rho_\tau(p)$ is the limit of the rotation numbers of Reeb orbits on smoothings of $X$ that converge to $\gamma$.

Fix matrices ${\mathbf i}, {\mathbf j}, {\mathbf k} \in \op{SO}(4)$ which represent the quaternion algebra, such that ${\mathbf i}$ is the standard almost complex structure. It follows from the formula $\omega_0(V,W)=\langle {\mathbf i}V,W\rangle$, together with the quaternion relations, that the matrices ${\mathbf i}$, ${\mathbf j}$, and ${\mathbf k}$ are symplectic. In examples below, in the coordinates $x_1,x_2,y_1,y_2$, we use the choice
\[
{\mathbf i} = \begin{pmatrix} & & -1 & \\ & & & -1 \\ 1 & & & \\ & 1 & &\end{pmatrix}, \quad {\mathbf j} = \begin{pmatrix} & -1 & & \\ 1 & & & \\ & & & 1 \\ & & -1 & \\ \end{pmatrix}, \quad {\mathbf k} = \begin{pmatrix} & & & -1 \\ & & 1 & \\ & -1 & & \\ 1 & & & \end{pmatrix}.
\]

\begin{definition}
\label{def:qtfg}
Let $X$ be a symplectic polytope in $\R^4$. We define the {\bf quaternionic trivialization\/} $(\tau,\widetilde{\phi})$ of the symplectic flow graph $G(X)$ as follows.
\begin{itemize}
  \item Let $F$ be a $2$-face of $X$. We define the isomorphism
  \[
\tau_F: TF \stackrel{\simeq}{\longrightarrow} \R^2
  \]
  as follows. By Lemma~\ref{lem:Reebcone}, there is a unique $3$-face $E$ adjacent to $F$ such that the Reeb cone $R_F^+$ consists of the nonnegative multiples of the Reeb vector field $R_E$, and the latter points into $E$ from $F$. Let $\nu$ denote the outward unit normal vector to $E$. If $V\in TF$, define
  \begin{equation}
  \label{eqn:tauF}
\tau_F(V) = (\langle V,{\mathbf j}\nu\rangle, \langle V,{\mathbf k}\nu\rangle).
  \end{equation}
  \item If $e$ is an edge from $F_1$ to $F_2$, define $\widetilde{\phi}_{e,\tau}\in\widetilde{\op{Sp}}(2)$ to be the unique lift of the symplectic matrix
  \begin{equation}
  \label{eqn:transitionmap}
\tau_{F_2} \circ T\phi_e \circ \tau_{F_1}^{-1} \in \op{Sp}(2)
  \end{equation}
  that has rotation number in the interval $(-1/2,1/2]$.
\end{itemize}
\end{definition}

The following lemma verifies that this is a legitimate trivialization.

\begin{lemma}
\label{lem:qtfg}
Let $X$ be a symplectic polytope in $\R^4$. If $F$ is a $2$-face of $X$, then the linear map $\tau_F$ in \eqref{eqn:tauF} is an isomorphism of symplectic vector spaces.
\end{lemma}

\begin{proof}
Let $E$ and $\nu$ be as in the definition of $\tau_F$. Then $\{{\mathbf i}\nu,{\mathbf j}\nu,{\mathbf k}\nu\}$ is an orthonormal basis for $TE$. We have $\omega_0({\mathbf i}\nu,{\mathbf j}\nu)=\omega_0({\mathbf i}\nu,{\mathbf k}\nu)=0$ and $\omega_0({\mathbf j}\nu,{\mathbf k}\nu)=1$. If $V$ and $W$ are any two vectors in $TF\subset TE$, then expanding them in this basis, we find that $\omega_0(V,W) = \omega_0(\tau_F(V),\tau_F(W))$.
\end{proof}

\begin{remark}
\label{rem:altcon}
An alternate convention for the quaternionic trivialization would be to define an isomorphism
 \[
\tau'_F: TF \stackrel{\simeq}{\longrightarrow} \R^2
\]
as follows. Let $E'$ be the other $3$-face adjacent to $F$ (so that the Reeb vector field $R_{E'}$ points out of $E$ along $F$), and let $\nu'$ denote the outward unit normal vector to $E'$. Define
\[
\tau'_F(V) = (\langle V,{\mathbf j}\nu'\rangle, \langle V,{\mathbf k}\nu'\rangle).
\]
This is also an isomorphism of symplectic vector spaces by the same argument as in Lemma~\ref{lem:qtfg}.
\end{remark}

\begin{definition}
\label{def:transitionmatrix}
If $X$ is a symplectic polytope in $\R^4$ and $F$ is a $2$-face of $X$, define the {\bf transition matrix\/}
\[
\psi_F = \tau_F\circ (\tau_F')^{-1}\in\op{Sp}(2).
\]
\end{definition}

\begin{lemma}
\label{lem:transitionmatrix}
If $X$ is a symplectic polytope in $\R^4$ and $F$ is a $2$-face of $X$, then the transition matrix $\psi_F$ is positive elliptic (see Definition~\ref{def:classifySp2}).
\end{lemma}

\begin{proof}
We compute that
\begin{equation}
\label{eqn:tauFprimeinverse}
(\tau'_F)^{-1} = \left( {\mathbf j}\nu' - \frac{\langle {\mathbf j}\nu',\nu\rangle}{\langle {\mathbf i}\nu',\nu\rangle}{\mathbf i}\nu', {\mathbf k}\nu' - \frac{\langle {\mathbf k}\nu',\nu\rangle}{\langle {\mathbf i}\nu',\nu\rangle}{\mathbf i}\nu'\right).
\end{equation}
To simplify notation, write $a_1 = \langle \nu',\nu\rangle$, $a_2=\langle {\mathbf i}\nu',\nu\rangle$, $a_3=\langle {\mathbf j}\nu',\nu\rangle$, and $a_4=\langle {\mathbf k}\nu',\nu\rangle$. It then follows from \eqref{eqn:tauF} and \eqref{eqn:tauFprimeinverse} that
\[
\psi_F = \frac{1}{a_2} \begin{pmatrix}
a_1a_2 - a_3a_4 & -a_2^2 -a_4^2\\ a_2^2 + a_3^2 & a_1a_2 + a_3a_4 
\end{pmatrix}
\]
Then $\op{Tr}(\psi_F) = 2\langle \nu',\nu\rangle \in (-2,2)$, so $\psi_F$ is elliptic. Moreover $a_2>0$ by Lemma~\ref{lem:EinEout} below, so $\psi_F$ is positive elliptic.
\end{proof}

\begin{corollary}
\label{cor:rotrange}
If $E$ is a $3$-face of $X$, if $F_1$ and $F_2$ are $2$-faces of $X$, and if there is a trajectory of the Reeb vector field on $E$ from some point in $F_1$ to some point in $F_2$, then $\widetilde{\phi}_{e,\tau}$ has rotation number in the interval $(0,1/2)$.
\end{corollary}

\begin{proof}
It follows from the definitions that the map \eqref{eqn:transitionmap} agrees with the transition matrix $\psi_{F_2}$. By Lemma~\ref{lem:transitionmatrix}, this matrix is positive elliptic. It then follows from Lemma~\ref{lem:compute_rho_bar} that its mod $\Z$ rotation number is in the interval $(0,1/2)$.
\end{proof}

\begin{definition}
\label{def:crn}
Let $X$ be a symplectic polytope in $\R^4$. Let $\gamma$ be a Type 1 combinatorial Reeb orbit for $X$.
\begin{itemize}
\item
We define the {\bf combinatorial rotation number\/} of $\gamma$ by
\[
\rho_{\op{comb}}(\gamma) = \rho_\tau(p),
\]
where $(p,x)$ is the periodic orbit of $G(X)$ corresponding to $\gamma$ in Proposition~\ref{prop:orbitbijection}, and $\tau$ is the quaternionic trivialization of $X$.
\item We say that $\gamma$ is {\bf nondegenerate\/} if the periodic orbit $(p,x)$ is nondegenerate as in Definition~\ref{def:degenerate_orbit}. In this case we define the {\bf combinatorial Conley-Zehnder index\/} of $\gamma$ by equation \eqref{eqn:ccz}.
\end{itemize}
\end{definition}

\begin{remark}
\label{rem:ucmult}
By Corollary~\ref{cor:rotrange}, the combinatorial rotation number is the rotation number of a product of elements of $\widetilde{\op{Sp}}(2)$ each with rotation number in the interval $(0,1/2)$. A formula for computing the rotation number of such a product is given by Proposition~\ref{prop:ucmult}. 
\end{remark}

\subsection{Example: the 24-cell}
\label{sec:24_cell}

We now compute the symplectic flow graph $G(X)=(\Gamma,A,\omega,\Phi)$ and the quaternionic trivialization $\tau$ for the example where $X$ is the $24$-cell with vertices
\[
(\pm1,0,0,0),(0,\pm1,0,0),(0,0,\pm1,0),(0,0,0,\pm1),(\pm1/2,\pm1/2,\pm1/2,\pm1/2).
\]

The polytope $X$ has $24$ three-faces, each of which is an octahedron. The $3$-faces are contained in the hyperplaces
\[
\pm x_1 \pm x_2 = 1,\; \pm x_1 \pm y_1 = 1,\; \pm x_1 \pm y_2 = 1, \; \pm x_2 \pm y_1 = 1, \; \pm x_2 \pm y_2 = 1, \; \pm y_1 \pm y_2 = 1.
\]
There are $96$ two-faces, each of which is a triangle; thus the graph $\Gamma$ has $96$ vertices. It follows from the calculations below that none of the $2$-faces is Lagrangian, so that $X$ is a symplectic polytope.

To understand the edges of the graph $\Gamma$, consider for example the $3$-face $E$ contained in the hyperplane $x_1+y_1=1$. The vertices of this $3$-face are
\[
(1,0,0,0), (1/2,\pm1/2,1/2,\pm1/2), (0,0,1,0).
\]
The unit normal vector to this face is
\[
\nu = \frac{1}{\sqrt{2}}(1,0,1,0).
\]
The Reeb vector field on $E$ is
\[
R_E = 2\left(-\frac{\partial}{\partial x_1} + \frac{\partial}{\partial y_1}\right).
\]
Thus the Reeb flow on $E$ flows from the vertex $(1,0,0,0)$ to the vertex $(0,0,1,0)$ in time $1/2$. Each of the four $2$-faces of $E$ adjacent to $(1,0,0,0)$ flows to one of the four $2$-faces of $E$ adjacent to $(0,0,1,0)$, by an affine linear isomorphism.

For example, let $F_1$ be the $2$-face with vertices $(1,0,0,0)$, $(1/2,1/2,1/2,\pm 1/2)$, and let $F_2$ be the $2$-face with vertices $(0,0,1,0)$, $(1/2,1/2,1/2,\pm 1/2)$. Then $F_1$ flows to $F_2$, so there is an edge $e$ in the graph $\Gamma$ from $F_1$ to $F_2$. More explicitly, we can parametrize $F_1$ as
\[
\left(1-\frac{t_1+t_2}{2}, \frac{t_1+t_2}{2}, \frac{t_1+t_2}{2}, \frac{t_1-t_2}{2}\right), \quad t_1,t_2>0, \; t_1+t_2<1,
\]
and we can parametrize $F_2$ as
\[
\left(\frac{t_1+t_2}{2}, \frac{t_1+t_2}{2}, 1 - \frac{t_1+t_2}{2}, \frac{t_1-t_2}{2}\right), \quad t_1,t_2>0, \; t_1+t_2<1.
\]
With respect to these parametrizations, the flow map $\phi_e$ is simply
\[
\phi_e(t_1,t_2) = (t_1,t_2).
\]
The domain $D_e$ of $\phi_e$ is all of $F_1$, and the action function is
\[
f_e(t_1,t_2) = \frac{1-t_1-t_2}{2}.
\]

It turns out that for every other $3$-face $E'$, there is a linear symplectomorphism $A$ of $\R^4$ such that $AX=X$ and $AE=E'$. In fact, we can take $A$ to be right multiplication by an appropriate unit quaternion. It follows from this symplectic symmetry that the Reeb flow on each $3$-face behaves analogously. Putting these Reeb flows together, one finds that the graph $\Gamma$ consists of $8$ disjoint $12$-cycles. (This example is highly non-generic!) Further calculations show that for each $12$-cycle $p$, the map $\phi_p$ is the identity, so that every point in the interior of a $2$-face is on a Type 1 combinatorial Reeb orbit. Moreover, the action of each such orbit is equal to $2$. In particular, $X$ is ``combinatorially Zoll'' in the sense of Definition~\ref{def:combinatorially_Zoll}. Also, the volume of $X$ is $2$, so $X$ has systolic ratio $1$.

To see how the quaternionic trivialization works, let us compute $\widetilde{\phi}_{e,\tau}$ for the edge $e$ above. For the $2$-face $F_1$ above, the isomorphism $\tau_{F_1}$ is given in terms of the unit normal vector $\nu$ to $E$. We compute that
\[
{\mathbf j}\nu = \frac{1}{\sqrt{2}}(0,1,0,-1),\quad\quad
{\mathbf k}\nu = \frac{1}{\sqrt{2}}(0,1,0,1).
\]
It follows that in terms of the basis $(\partial_{t_1},\partial_{t_2})$ for $TF_1$, we have
\[
\tau_{F_1} = \frac{1}{\sqrt{2}}\begin{pmatrix} 0 & 1 \\ 1 & 0 \end{pmatrix}.
\]
For the $2$-face $F_2$ above, the isomorphism $\tau_{F_2}$ is given in terms of the unit normal vector to the {\em other\/} $3$-face adjacent to $F_2$. This other $3$-face is in the hyperplane $x_2+y_1=1$ and so has unit normal vector
\[
\nu' = \frac{1}{\sqrt{2}}(0,1,1,0).
\]
We then similarly compute that in terms of the basis $(\partial_{t_1},\partial_{t_2})$ for $TF_2$, we have
\[
\tau_{F_2} = \frac{1}{\sqrt{2}}\begin{pmatrix} -1 & 0 \\ 1 & 1 \end{pmatrix}
\]
Therefore the matrix \eqref{eqn:transitionmap} for the edge $e$ is
\[
\tau_{F_2} \circ T\phi_e\circ \tau_{F_1}^{-1} = \begin{pmatrix} -1 & 0 \\ 1 & 1 \end{pmatrix} \begin{pmatrix} 0 & 1 \\ 1 & 0 \end{pmatrix}^{-1} = \begin{pmatrix} 0 & -1 \\ 1 & 1 \end{pmatrix}.
\]
This matrix is positive elliptic and has eigenvalues $e^{\pm i\pi/3}$. It follows that its lift $\widetilde{\phi}_{e,\tau}$ in $\widetilde{\op{Sp}}(2)$ has rotation number $1/6$.

For one of the other three edges associated to $E$, the matrix \eqref{eqn:transitionmap} is the same as above, and for the other two edges associated to $E$, the matrix is $\begin{pmatrix} 1 & -1 \\ 1 & 0\end{pmatrix}$, whose lift also has rotation number $1/6$. It then follows from the quaternionic symmetry of $X$ mentioned earlier that for every edge $e'$ of the graph $\Gamma$, the lift $\widetilde{\phi}_{e',\tau}$ is one of the above two matrices with rotation number $1/6$. One can further check that for each $12$-cycle in the graph, one obtains just one of the above two matrices repeated $12$ times, so each corresponding Type 1 combinatorial Reeb orbit has rotation number equal to $2$.

%% file: s3_reeb_dynamics_on_polytopes.tex
\section{Reeb dynamics on symplectic polytopes}
\label{sec:rdsp} 

The goal of this section is to prove Proposition~\ref{prop:well-posed} and  Lemma~\ref{lem:Reebcone}, describing the Reeb dynamics on the boundary of a symplectic polytope in $\R^4$.

\subsection{Preliminaries on tangent and normal cones}

We now prove some lemmas about tangent and normal cones which we will need; see \S\ref{sec:cro} for the definitions.

Recall that if $C$ is a cone in $\R^m$, its {\bf polar dual\/} is defined by
\[
C^o=\{y\in\R^m\mid \langle x,y\rangle \le 0 \;\forall x\in C\}.
\]

\begin{lemma}
\label{lem:ntd}
Let $X$ be a convex set in $\R^m$ and let $y\in\partial X$. Then
\[
N_y^+X = (T_y^+X)^o, \quad\quad T_y^+X = (N_y^+X)^o.
\]
\end{lemma}

\begin{proof}
If $C$ is a closed cone then $(C^o)^o = C$, so it suffices to prove that $N_y^+X = (T_y^+X)^o$.

To show that $N_y^+X\subset (T_y^+X)^\circ$, let $v\in N_y^+X$ and $w\in T_y^+X$; we need to show that $\langle v,w\rangle \le 0$. By the definition of $T_y^+X$, there exist a sequence of vectors $\{w_i\}$ and a sequence of positive real numbers $\{\epsilon_i\}$ such that $y+\epsilon_iw_i\in X$ for each $i$ and $\lim_{i\to\infty}w_i=w$. By the definition of $N_y^+X$ we have $\langle v,w_i\rangle \le 0$, and so $\langle v,w\rangle \le 0$.

To prove the reverse inclusion, if $v\in (T_x^+X)^o$, then for any $x\in X$ we have $x-y\in T_y^+X$, so $\langle v,x-y\rangle \le 0$. It follows that $v\in N_y^+X$.
\end{proof}

If $X$ is a convex polytope in $\R^m$ and if $E$ is an $(m-1)$-face of $X$, let $\nu_E$ denote the outward unit normal vector to $E$.

\begin{lemma}
\label{lem:ncn}
Let $X$ be a convex polytope in $\R^m$ and let $F$ be a face of $X$. Let $E_1,\ldots, E_k$ denote the $(m-1)$-faces whose closures contain $F$. Then
\begin{align}
\label{eqn:TFcone}
T_F^+X &= \left\{w\in \R^m \mid \langle w,\nu_{E_i}\rangle \le 0 \;\; \forall i=1,\ldots,k \right\},
\\
\label{eqn:NFcone}
N_F^+X & = \op{Cone}
	\left(
			\nu_{E_1},\ldots, \nu_{E_k}
	\right).
\end{align}
\end{lemma}

\begin{proof}
Let $y\in F$, and let $B$ be a small ball around $y$. Then $B\cap X=\cap_i(B\cap H_i)$ where $\{H_i\}$ is the set of all defining half-spaces for $X$ whose boundaries contain $F$. The boundaries of the half-spaces $H_i$ are the hyperplanes that contain the $(m-1)$-faces $E_1,\ldots, E_k$. It follows that $B\cap X$  is the set of $x\in B$ such that $\langle x-y,\nu_{E_i}\rangle \le 0$ for each $i=1,\ldots, k$. Equation \eqref{eqn:TFcone} follows. Taking polar duals and using Lemma~\ref{lem:ntd} then proves \eqref{eqn:NFcone}.
\end{proof}

\begin{lemma}
\label{lem:pp1}
Let $X$ be a convex polytope in $\R^m$ and let $F$ be a face of $X$. Let $v\in N_F^+X\setminus\{0\}$ and let $w\in T_F^+X\setminus\{0\}$. Then $\langle v,w\rangle = 0$ if and only if there is a face $E$ of $X$ with $F\subset \overline{E}$ such that $v\in N_E^+X$ and $w\in T_F^+\overline{E}$.
\end{lemma}

Here if $E\neq F$ then $T_F^+\overline{E}$ denotes the tangent cone of the polytope $\overline{E}$ at the face $F$ of $\overline{E}$; if $E=F$, then we interpret $T_F^+\overline{E}=TF$.

\begin{proof}[Proof of Lemma~\ref{lem:pp1}.]
As in Lemma~\ref{lem:ncn}, let $E_1,\ldots, E_k$ denote the $(m-1)$-faces adjacent to $F$.

$(\Rightarrow)$
By the definitions of $N_F^+X$ and $T_F^+X$, if $v\in N_F^+X$ and $w\in T_F^+X$ then $\langle v,w\rangle \le 0$. Assume also that $v$ and $w$ are both nonzero and $\langle v,w\rangle = 0$. Then we must have $v\in\partial N_F^+X$ and $w\in\partial T_F^+X$; otherwise we could perturb $v$ or $w$ to make the inner product positive, which would be a contradiction.

Since $w\in\partial T_F^+X$, it follows from \eqref{eqn:TFcone} that $\langle w,\vu_{E_i}\rangle = 0$ for some $i$. By renumbering we can arrange that $\langle w,\nu_{E_i}\rangle = 0$ if and only if $i\le l$ where $1\le l\le k$. Let $E=\cap_{i=1}^l E_i$. Then $E$ is a face of $X$ adjacent to $F$, and $w\in T_F^+\overline{E}$.

We now want to show that $v\in N_E^+X$. By \eqref{eqn:NFcone}, we can write $v = \sum_{i=1}^k a_i\nu_{E_i}$ with $a_i\ge 0$. Since $\langle v,w\rangle = 0$ and $\langle w,\vu_{E_i}\rangle = 0$ for $i\le l$ and $\langle w,\nu_{E_i}\rangle < 0$ for $i>l$, we must have $a_i=0$ for $i>l$. Thus $v\in\op{Cone}(\vu_{E_1},\ldots,\nu_{E_l})$, so by \eqref{eqn:NFcone} again, $v\in N_F^+X$. 

$(\Leftarrow)$ Assume that there is a face $E$ adjacent to $X$ such that $v\in N_E^+X$ and $w\in T_F^+\overline{E}$. We can renumber so that $E=\cap_{i=1}^l E_i$ where $1\le l \le k$. Then $v\in\op{Cone}(\nu_{E_1},\ldots,\nu_{E_l})$, and $\langle w,\nu_{E_i}\rangle = 0$ for $i\le l$, so $\langle v,w\rangle = 0$.
\end{proof}

\subsection{The combinatorial Reeb flow is locally well-posed}
\label{sec:wp}

We now prove Proposition~\ref{prop:well-posed}, asserting that the ``combinatorial Reeb flow'' on the boundary of a symplectic polytope in $\R^4$ is locally well-posed. This is a consequence of the following two lemmas:

\begin{lemma}
\label{lem:wp1}
Let $X$ be a convex polytope in $\R^4$, and let $F$ be a face of $X$. Then the Reeb cone
\[
R_F^+X = {\mathbf i}N_F^+X\cap T_F^+X
\]
has dimension at least $1$.
\end{lemma}

Note that there is no need to assume that $0\in\op{int}(X)$ in the above lemma, because the Reeb cone is invariant under translation of $X$.

\begin{lemma}
\label{lem:wp2}
Let $X$ be a symplectic polytope in $\R^4$ and let $F$ be a face of $X$. Then the Reeb cone $R_F^+X$ has dimension at most $1$.
\end{lemma}

\begin{proof}[Proof of Lemma~\ref{lem:wp1}.] The proof has four steps.

{\em Step 1.\/} We need to show that there exists a unit vector in $R_F^+X$. We first rephrase this statement in a way that can be studied topologically.

Define
\[
B = \left\{(v,w)\in N_F^+X\times T_F^+X \;\big|\; \|v\|=\|w\|=1,\; \langle v,w\rangle = 0\right\}.
\]
Define a fiber bundle $\pi:Z\to B$ with fiber $S^2$ by setting
\[
Z_{(v,w)}=\left\{u\in\R^4 \;\big|\; \|u\|=1, \; \langle u,v\rangle = 0 \right\}.
\]
Define two sections
\[
s_0,s_1: B \longrightarrow Z
\]
by
\[
\begin{split}
s_0(v,w) &= {\mathbf i}v,\\
s_1(v,w) &= w.
\end{split}
\]
To show that there exists a unit vector in $R_F^+X$, we need to show that there exists a point $(v,w)\in B$ with $s_0(v,w) = s_1(v,w)$.

{\em Step 2.\/} 
Let
\[
B_0 = \left\{w\in\partial T_F^+X \; \big| \; \|w\|=1\right\}.
\]
The space $B_0$ is the set of unit vectors on the boundary of a nondegenerate cone, and thus is homeomorphic to $S^2$. Recall from the proof of Lemma~\ref{lem:pp1} that if $(v,w)\in B$ then $w\in B_0$. We now show that the projection $B\to B_0$ sending $(v,w)\mapsto w$ is a homotopy equivalence.

To do so, observe that by Lemma~\ref{lem:pp1}, we have
\begin{equation}
\label{eqn:Bunion}
B = \bigcup_{F\subset E} \left\{v\in N_E^+X\;\big|\; \|v\|=1\right\} \times \left\{w\in T_F^+\overline{E} \;\big|\; \|w\|=1\right\}.
\end{equation}

If $F$ is a $3$-face, then in the union \eqref{eqn:Bunion}, we only have $E=F$; there is a unique unit vector $v\in N_E^+X$, and so the projection $B\to B_0$ is a homeomorphism.

If $F$ is a $2$-face, then in \eqref{eqn:Bunion}, $E$ can be either $F$ itself, or one of the two three-faces adjacent to $F$, call them $E_1$ and $E_2$. The contribution from $E=F$ is a cylinder, while the contributions from $E=E_1$ and $E_2$ are disks which are glued to the cylinder along its boundary. The projection $B\to B_0$ collapses the cylinder to a circle, which again is a homotopy equivalence.

If $F$ is a $1$-face, with $k$ adjacent $3$-faces, then the contribution to \eqref{eqn:Bunion} from $E=F$ consists of two disjoint closed $k$-gons. Each $2$-face $E$ adjacent to $F$ contributes a square with opposite edges glued to one edge of each $k$-gon. Each $3$-face $E$ adjacent to $F$ contributes a bigon filling in the gap between two consecutive squares. The projection $B\to B_0$ collapses each $k$-gon to a point and each bigon to an interval, which again is a homotopy equivalence.

Finally, suppose that $F$ is a $0$-face. Then $E=F$ makes no contribution to \eqref{eqn:Bunion}, since $TF=\{0\}$ contains no unit vectors. Now $B_0$ has a cell decomposition consisting of a $k$-cell for each $(k+1)$-face adjacent to $F$. The space $B$ is obtained from $B_0$ by thickening each $0$-cell to a closed polygon, and thickening each $1$-cell to a square. Again, this is a homotopy equivalence.

{\em Step 3.\/} The $S^2$-bundle $Z\to B$ is trivial. To see this, observe that $Z$ is the pullback of a bundle over $N_F^+X\setminus\{0\}$, whose fiber over $v$ is the set of unit vectors orthogonal to $v$. Since $N_F^+X\setminus\{0\}$ is contractible, the latter bundle is trivial, and thus so is $Z$. In particular, the bundle $Z$ has two homotopy classes of trivialization, which differ only in the orientation of the fiber. We now show that, using a trivialization to regard $s_0$ and $s_1$ as maps $B\to S^2$, the mod $2$ degrees of these maps are given by $\op{deg}(s_0)=0$ and $\op{deg}(s_1)=1$.

It follows from the triviality of the bundle $Z$ that $\op{deg}(s_0)=0$.

To prove that $\op{deg}(s_1)=1$, we need to pick an explicit trivialization of $Z$. To do so, fix a vector $v_0\in\op{int}(T_F^+X)$. Let $S$ denote the set of unit vectors in the orthogonal complement $v_0^\perp$. Let $P:\R^4\to v_0^\perp$ denote the orthogonal projection. We then have a trivialization
\[
Z \stackrel{\simeq}{\longrightarrow} B\times S
\]
sending
\[
((v,w),u) \longmapsto ((v,w),Pu/\|Pu\|).
\]
Note here that for every $(v,w)\in B$, the restriction of $P$ to $v^\perp$ is an isomorphism, because otherwise $v$ would be orthogonal to $v_0$, but in fact we have $\langle v,v_0\rangle < 0$.

With respect to this trivialization, the section $s_1$ is a map $B\to S$ which is the composition of the projection $B\to B_0$ with the map $B_0\to S$ sending
\[
w \longmapsto Pw/\|Pw\|.
\]
The former map is a homotopy equivalence by Step 2, and the latter map is a homeomorphism because $v_0$ is not parallel to any vector in $\partial T_F^+X$.  Thus $\op{deg}(s_1)=1$.

{\em Step 4.\/} We now complete the proof of the lemma. Suppose to get a contradiction that there does not exist a point $p\in B$ with $s_0(p)=s_1(p)$. It follows, using a trivialization of $Z$ to regard $s_0$ and $s_1$ as maps $B\to S^2$, that $s_1$ is homotopic to the composition of $s_0$ with the antipodal map. Then $\op{deg}(s_1)=-\op{deg}(s_0)$. This contradicts Step 3.
\end{proof}

\begin{remark}
It might be possible to generalize Lemma~\ref{lem:wp1} to show that if $X$ is any convex set in $\R^{2n}$ with nonempty interior and if $z\in\partial X$, then the Reeb cone $R_z^+X$ is at least one dimensional.
\end{remark}

We now prepare for the proof of Lemma~\ref{lem:wp2}.

\begin{lemma} \label{lem:weak_well_posedness_for_polytopes} Let $X$ be a convex polytope in $\R^{2n}$. Then for every face $F$ of $X$, there exists a face $E$ with $F\subset\overline{E}$ such that
\[
R_F^+X \subset T^+_F\bar{E}.
\]
\end{lemma}

\noindent
{\em Proof.\/}
Let $\{E_i\}_{i=1}^N$ denote the set of faces whose closures contain $F$. By Lemma~\ref{lem:pp1}, we have
\begin{equation}
\label{eqn:wwp}
R_F^+X \subset \bigcup_{i=1}^N T_F^+\bar{E}_i.
\end{equation}

Let $V$ denote the subspace of $\R^{2n}$ spanned by $R^+_FX$. Note that since the latter set is a cone, it has a nonempty interior in $V$. We claim now that $V\subset TE_i$ for some $i$. If not, then $V\cap TE_i$ is a proper subspace of $V$ for each $i$. But by \eqref{eqn:wwp}, we have
\[
R^+_FX = \left(\cup_i T^+_F\bar{E}_i\right) \cap R^+_FX \subset \left(\cup_i TE_i\right) \cap V.
\]
This is a contradiction, since the left hand side has a nonempty interior in $V$, while the right hand side is a union of proper subspaces of $V$.

Since $V \subset TE_i$, it follows that $R^+_FX \subset T^+_F\bar{E}_i$, because by \eqref{eqn:wwp} again,
\[
R^+_FX = R^+_FX \cap V = R^+_FX \cap TE_i
\]
\[
\hspace{1.8in}
\subset  TE_i \cap \bigg(\bigcup_j T^+_F\bar{E}_j\bigg) = T_F\bar{E}_i, \hspace{1.8in}\Box
\]

\begin{lemma}
\label{lem:wwp2}
Let $X$ be a convex polytope in $\R^{2n}$, and let $F$ be a face of $X$. Let $v\in R_F^+X$. Suppose that $v\in\op{int}(T_F^+\overline{E})$ for some $(2n-1)$-face $E$ whose closure contains $F$. Then $v$ is a positive multiple of ${\mathbf i}\nu_E$.
\end{lemma}

\begin{proof}
Let $E=E_1,\ldots,E_N$ denote the $(2n-1)$-faces whose closures contain $F$, and let $\nu_i$ denote the outward unit normal vector to $E$. Since $v\in\op{int}(T_F^+\overline{E})$, we have $\langle v,\nu_1\rangle=0$ and $\langle v,\nu_i\rangle < 0$ for $i>1$. Since $-{\mathbf i}v\in N_F^+X$, it follows from Lemma~\ref{lem:ncn} that we can write
\[
-{\mathbf i}v = \sum_{i=1}^Na_i\nu_i
\]
with $a_i\ge 0$. Since $\langle v,{\mathbf i}v\rangle = 0$, we conclude that $a_i=0$ for $i>1$. Thus $-{\mathbf i}v=a_1\nu_1$, and $a_1>0$.
\end{proof}

\begin{proof}[Proof of Lemma~\ref{lem:wp2}.]
Suppose $v_0,v_1$ are distinct unit vectors in $R_F^+X$. By Lemma~\ref{lem:weak_well_posedness_for_polytopes}, there is a $3$-face $E$ such that $v_0$ and $v_1$ are both in $T_F^+\bar{E}$. In particular, $v_1$ and $v_2$ are linearly independent.

Since $v_0$ and $v_1$ are both in the cone $R_F^+X$, it follows that if $t\in[0,1]$ then the affine linear combination $(1-t)v_0+tv_1$ is also in this cone. Since $v_0$ and $v_1$ are linearly independent, these affine linear combinations cannot be in the interior of $T_F^+\overline{E}$, or else this would contradict the projective uniqueness in Lemma~\ref{lem:wwp2}. Consequently $v_0$ and $v_1$ are both contained in $T_F^+\overline{E'}$ for some $2$-face $E'$ on the boundary of $\overline{E}$.

We now have
\[
\omega(v_0,v_1) = \langle v_0,-{\mathbf i}v_1\rangle \le 0,
\]
where the inequality holds since $v_0\in T_F^+X$ and $-{\mathbf i}v_1\in N_F^+X$. By a symmetric calculation, $\omega(v_1,v_0)\le 0$. It follows that $\omega(v_0,v_1)=0$. Since $v_0$ and $v_1$ are linearly independent vectors in $TE'$, this contradicts the hypothesis that $\omega|_{TE'}$ is nondegenerate.
\end{proof}

\subsection{Description of the Reeb cone}
\label{sec:drc}

We now prove Lemma~\ref{lem:Reebcone}, describing the possibilities for the Reeb cone of a face of a symplectic polytope in $\R^4$.

\begin{lemma}
\label{lem:EinEout}
Let $X$ be a convex polytope in $\R^4$ and let $F$ be a $2$-face of $X$. Let $E_1$ and $E_2$ denote the $3$-faces adjacent to $F$, and let $\nu_i$ denote the outward unit normal vector to $E_i$.
\begin{itemize}
	\item[\emph{(a)}] If $\langle {\mathbf i}\nu_{1},\nu_{2}\rangle < 0$, then every nonzero vector $w$ in the Reeb cone $R_{E_1}^+$ points into $E_1$ from $F$, that is $w\in \op{int}(T_F^+\overline{E_1})$.
	\item[\emph{(b)}] If $\langle {\mathbf i}\nu_{1},\nu_{2}\rangle > 0$, then every nonzero vector $w$ in the Reeb cone $R_{E_1}^+$ points out of $E_1$ from $F$, that is $w\in \op{int}(-T_F^+\overline{E_1})$.
	\item[\emph{(c)}] If $\langle {\mathbf i}\nu_{1},\nu_{2}\rangle = 0$, then $F$ is Lagrangian.
\end{itemize}
\end{lemma}

\begin{proof}
Let $\eta$ denote the unit normal vector to $F$ in $T\overline{E_1}$ pointing into $E_1$. The vector $\eta$ must be a linear combination of $\nu_1$ and $\nu_2$ (since it is normal to $F$), it must be orthogonal to $\nu_1$ (since it is tangent to $E_1$), and it must have negative inner product with $\nu_2$ (since it points into $E_1$). It follows that
\begin{equation}
\label{eqn:eta}
\eta = \frac{-\nu_2 + \langle\nu_1,\nu_2\rangle\nu_1}{\|-\nu_2 + \langle\nu_1,\nu_2\rangle\nu_1\|}.
\end{equation}

The vector $w$ points into $E_1$ if and only if $\langle \eta,w\rangle >0$, and the vector $w$ points out of $E_1$ if and only if $\langle \eta,w\rangle < 0$. For $w$ in the Reeb cone of $E_1$, we know that $w$ is a positive multiple of ${\mathbf i}\nu_1$. By equation \eqref{eqn:eta}, we have
\[
\langle \eta,{\mathbf i}\nu_1\rangle = \frac{-\langle{\mathbf i}\nu_1,\nu_2\rangle}{\|-\nu_2 + \langle\nu_1,\nu_2\rangle\nu_1\|}.
\]
Thus if $\langle {\mathbf i}\nu_1,\nu_2\rangle$ is nonzero, then it has opposite sign from $\langle \eta,w\rangle$. This proves (a) and (b).

If $\langle {\mathbf i}\nu_1,\nu_2\rangle = 0$, then $\omega({\mathbf i}\nu_1,{\mathbf i}\nu_2) = 0$, but ${\mathbf i}\nu_1$ and ${\mathbf i}\nu_2$ are linearly independent tangent vectors to $F$, so $F$ is Lagrangian. This proves (c).
\end{proof}

\begin{lemma}
\label{lem:la}
Let $X$ be a convex polytope in $\R^4$ and let $F$ be a 2-face of $X$. If $TF\cap R_F^+X\neq\{0\}$, then $F$ is Lagrangian.
\end{lemma}

\begin{proof}
If $w\in TF\cap R_F^+X$, then for any other vector $u\in TF$, we have
\[
\omega(w,u) = \langle {\mathbf i}w,u\rangle = 0
\]
since $-{\mathbf i}w\in N_F^+X$. If we also have $w\neq 0$, then it follows that $F$ is Lagrangian.
\end{proof}

\begin{proof}[Proof of Lemma~\ref{lem:Reebcone}.]
If $F$ is a $3$-face, then by the definition of the Reeb cone, $R_F^+X$ consists of all nonnegative multiples of ${\mathbf i}\nu_F$; and ${\mathbf i}\nu_F$ is a positive multiple of the Reeb vector field on $F$ by equation \eqref{eqn:Reebinu}.

Suppose now that $F$ is a $k$-face with $k<3$, and that $w$ is a nonzero vector in the Reeb cone $R_F^+X$. Applying Lemma~\ref{lem:pp1} to $v=-{\mathbf i}w$ and $w$, we deduce that there is a face $E$ of $X$ with $F\subset \overline{E}$ such that $-{\mathbf i}w\in N_E^+X$ and $w\in T_F^+\overline{E}$. In particular,
\begin{equation}
\label{eqn:terex}
w\in TE\cap R_E^+X.
\end{equation}

By Lemma~\ref{lem:la} and our hypothesis that $X$ is a symplectic polytope, $E$ is not a $2$-face.

If $F$ is a $2$-face, we conclude that $w$ is in the Reeb cone $R_E^+X$ for one of the $3$-faces $E$ adjacent to $F$. By Lemma~\ref{lem:EinEout}, $w$ must point into $E$.

If $F$ is a $1$-face, then $E$ is either a $3$-face adjacent to $F$, or $F$ itself. In the case when $E=F$, the vector $w$ cannot be in the Reeb cone of any $3$-face $F_3$ adjacent to $F$. The reason is that if $F_2$ is one of the two $2$-faces with $F \subset \overline{F_2} \subset \overline{F_3}$, then by Lemma~\ref{lem:EinEout}, the Reeb cone of $F_3$ is not tangent to $F_2$, so it certainly cannot be tangent to $F$.

If $F$ is a $0$-face, then $E$ is adjacent to $F$ and is either a $3$-face or a $1$-face. If $E$ is a $1$-face, then it is a bad $1$-face by \eqref{eqn:terex}.
\end{proof}

%% file: s4_quaternionic.tex
\section{The quaternionic trivialization}
\label{sec:quaternionic}

In this section let $Y\subset\R^4$ be a smooth star-shaped hypersurface with the contact form $\lambda = \lambda_0|_Y$ and contact structure $\xi=\Ker(\lambda)$. We now define a special trivialization $\tau$ of the contact structure $\xi$, and we prove a key property of this trivialization.

\subsection{Definition of the quaternionic trivialization}

The following definition is a smooth analogue of Definition~\ref{def:qtfg}.

\begin{definition}
Define the {\bf quaternionic trivialization}
\begin{equation}
\label{eqn:qt}
\tau: \xi \stackrel{\simeq}{\longrightarrow} Y\times \R^2
\end{equation}
as follows. If $y\in Y$ and $V\in T_yY$, let $\nu$ denote the outward unit normal to $Y$ at $y$, and define
\[
\tau(V) = \left(y,\langle V,{\mathbf j}\nu\rangle, \langle V,{\mathbf k}\nu\rangle\right).
\]
By abuse of notation, for fixed $y\in Y$ we write $\tau:\xi_y\stackrel{\simeq}{\longrightarrow} \R^2$ to denote the restriction of \eqref{eqn:qt} to $\xi_y$ followed by projection to $\R^2$.
\end{definition}

From now on we always use the quaternionic trivialization $\tau$ for smooth star-shaped hypersurfaces in $\R^4$.

\begin{lemma}
The quaternionic trivialization $\tau$ is a symplectic trivialization of $\xi$.
\end{lemma}

\begin{proof}
Same calculation as the proof of Lemma~\ref{lem:qtfg}(a).
\end{proof}

\begin{remark}
The inverse
\[
\tau^{-1}: Y\times \R^2\stackrel{\simeq}{\longrightarrow} \xi
\]
is described as follows. Recall from \eqref{eqn:Reebinu} that the Reeb vector field at $y$ is a positive multiple of ${\mathbf i}\nu$. Then $\tau^{-1}(y,(1,0))$ is obtained by projecting ${\mathbf j}\nu$ to $\xi_y$ along the Reeb vector field, while $\tau^{-1}(y,(0,1))$ is obtained by projecting ${\mathbf k}\nu$ to $\xi_y$ along the Reeb vector field. 
\end{remark}

\subsection{Linearized Reeb flow}

We now make some definitions which we will need in order to bound the rotation numbers of Reeb orbits and Reeb trajectories.

\begin{definition}
\label{def:linearized}
If $y\in Y$ and $t\ge 0$, define the {\bf linearized Reeb flow\/} $\phi(y,t)\in\op{Sp}(2)$ to be the composition
\begin{equation}
\label{eqn:lrf}
\R^2 \stackrel{\tau^{-1}}{\longrightarrow} \xi_y \stackrel{d\Phi_t}{\longrightarrow} \xi_{\Phi_t(y)} \stackrel{\tau}{\longrightarrow} \R^2
\end{equation}
where $\Phi_t:Y\to Y$ denotes the time $t$ flow of the Reeb vector field, and $\tau$ is the quaternionic trivialization. Define the {\bf lifted linearized Reeb flow\/} $\widetilde{\phi}(y,t)\in\widetilde{\op{Sp}}(2)$ to be the arc
\begin{equation}
\label{eqn:llrf}
\widetilde{\phi}(y,t) = 
\{\phi(y,s)\}_{s\in [0,t]}.
\end{equation}
\end{definition}

Note that we have the composition property
\[
\widetilde{\phi}(y,t_2+t_1) = \widetilde{\phi}(\phi_{t_1}(y),t_2) \circ \widetilde{\phi}(y,t_1).
\]

Next, let ${\mathbb P}\xi$ denote the ``projectivized'' contact structure
\[
{\mathbb P}\xi = (\xi\setminus Z)/\sim
\]
where $Z$ denotes the zero section, and two vectors are declared equivalent if they differ by multiplication by a positive scalar. Thus ${\mathbb P}\xi$ is an $S^1$-bundle over $Y$. The Reeb vector field $R$ on $Y$ canonically lifts, via the linearized Reeb flow, to a vector field $\widetilde{R}$ on ${\mathbb P}\xi$.

The quaternionic trivialization $\tau$ defines a diffeomorphism
\[
\overline{\tau}: {\mathbb P}\xi \stackrel{\simeq}{\longrightarrow} Y\times S^1.
\]
Let
\[
\sigma: {\mathbb P}\xi\longrightarrow S^1
\]
denote the composition of $\overline{\tau}$ with the projection $Y\times S^1\to S^1$.

\begin{definition}
Define the {\bf rotation rate\/}
\[
r: {\mathbb P}\xi\longrightarrow \R
\]
to be the derivative of $\sigma$ with respect to the lifted linearized Reeb flow,
\[
r=\widetilde{R}\sigma.
\]
Define the {\bf minimum rotation rate\/}
\[
r_{\op{min}}:Y\longrightarrow \R
\]
by
\[
r_{\op{min}}(y) = \min_{\widetilde{y}\in{\mathbb P}\xi_y}r(\widetilde{y}).
\]
\end{definition}

It follows from \eqref{eqn:rminbound} and \eqref{eqn:rhor} that we have the following lower bound on the rotation number of the lifted linearized flow of a Reeb trajectory.

\begin{lemma}
\label{lem:minrot}
Let $y$ be a smooth star-shaped hypersurface in $\R^4$, let $y\in Y$, and let $t\ge 0$. Then
\[
\rho(\widetilde{\phi}(y,t)) \ge \int_0^tr_{\op{min}}(\Phi_s(y))ds.
\]
\end{lemma}

\subsection{The curvature identity}

We now prove a key identity which relates the linearized Reeb flow, with respect to the quaternionic trivialization $\tau$, to the curvature of $Y$. This identity (in different notation) is due to U.\ Hryniewicz and P.\ Salom\~ao \cite{umberto}. Below, let $S:TY\tensor TY\to\R$ denote the second fundamental form defined by
\[
S(u,w) = \langle \nabla_u\nu,w\rangle,
\]
where $\nu$ denotes the outward unit normal vector to $Y$, and $\nabla$ denotes the trivial connection on the restriction of $T\R^4$ to $Y$. Also write $S(u)=S(u,u)$.

\begin{proposition}
\label{prop:uj}
Let $Y$ be a smooth star-shaped hypersurface in $\R^4$, let $y\in Y$, let $\theta\in\R/2\pi\Z$, and write $\sigma = \theta/2\pi\in\R/\Z$. Then at the point $\overline{\tau}^{-1}(y,\sigma)\in{\mathbb P}\xi$, we have
\begin{equation}
\label{eqn:curvatureidentity}
\widetilde{R}\sigma = \frac{1}{\pi\langle\nu,y\rangle}\left(S({\mathbf i}\nu) + S(\cos(\theta){\mathbf j}\nu + \sin(\theta){\mathbf k}\nu)\right).
\end{equation}
\end{proposition}

\begin{proof}
It follows from the definitions that
\begin{equation}
\label{eqn:ffd}
\begin{split}
2\pi\widetilde{R}\sigma =& \left\langle \mc{L}_R((\cos\theta){\mathbf j}\nu + (\sin\theta){\mathbf k}\nu), (\sin\theta){\mathbf j}\nu - (\cos\theta){\mathbf k}\nu\right\rangle\\
=& -(\cos^2\theta)\langle \mc{L}_R{\mathbf j}\nu,{\mathbf k}\nu\rangle + (\sin^2\theta)\langle \mc{L}_R{\mathbf k}\nu,{\mathbf j}\nu\rangle\\
& +(\sin\theta\cos\theta)(\langle \mc{L}_R{\mathbf j}\nu,{\mathbf j}\nu\rangle - \langle \mc{L}_R{\mathbf k}\nu,{\mathbf k}\nu\rangle).
\end{split}
\end{equation}
We compute
\begin{align}
\nonumber
\langle \mc{L}_R{\mathbf j}\nu,{\mathbf k}\nu\rangle &= \langle \nabla_R{\mathbf j}\nu - \nabla_{{\mathbf j}\nu}R, {\mathbf k}\nu\rangle\\
\nonumber
&= \frac{2}{\langle \nu,y\rangle}\left(\langle \nabla_{{\mathbf i}\nu}{\mathbf j}\nu,{\mathbf k}\nu\rangle - \langle\nabla_{{\mathbf j}\nu}{\mathbf i}\nu, {\mathbf k}\nu\rangle\right)\\
\nonumber
&= \frac{2}{\langle \nu,y\rangle}\left(-\langle \nabla_{{\mathbf i}\nu}\nu,{\mathbf i}\nu\rangle -\langle \nabla_{{\mathbf j}\nu}\nu,{\mathbf j}\nu\rangle\right)\\
\label{eqn:Ljk}
&= \frac{2}{\langle \nu,y\rangle}\left(-S({\mathbf i}\nu) - S({\mathbf j}\nu)\right).
\end{align}
Here in the second to third lines we have used the fact that multiplication on the left by a constant unit quaternion is an isometry. Similar calculations show that
\begin{align}
\label{eqn:Lkj}
\langle \mc{L}_R{\mathbf k}\nu,{\mathbf j}\nu\rangle &= \frac{2}{\langle \nu,y\rangle}\left(S({\mathbf i}\nu) + S({\mathbf k}\nu)\right),\\
\label{eqn:Ljj}
\langle \mc{L}_R{\mathbf j}\nu,{\mathbf j}\nu\rangle = - \langle \mc{L}_R{\mathbf k}\nu,{\mathbf k}\nu\rangle &= \frac{2}{\langle\nu,y\rangle} S({\mathbf j}\nu,{\mathbf k}\nu).
\end{align}
Plugging \eqref{eqn:Ljk}, \eqref{eqn:Lkj} and \eqref{eqn:Ljj} into \eqref{eqn:ffd} proves the curvature identity \eqref{eqn:ffd}.
\end{proof}

\begin{remark}
Since the second fundamental form is positive definite when $Y$ is strictly convex, and positive semidefinite when $Y$ is convex, by Lemma~\ref{lem:minrot} we obtain the following corollary:
{\em If $Y$ is a convex star-shaped hypersurface in $\R^4$ then $\widetilde{R}\sigma\ge 0$ everywhere, so $\widetilde{\phi}(y,t)$ has nonnegative rotation number for all $y\in Y$ and $t\ge 0$.
If $Y$ is a strictly convex star-shaped hypersurface in $\R^4$ then $\widetilde{R}\sigma >0$ everywhere, so $\widetilde{\phi}(y,t)$ has positive rotation number for all $y\in Y$ and $t>0$.\/}
\end{remark}

%% file: s5_dynamics_on_smoothings.tex
\section{Reeb dynamics on smoothings of polytopes}
\label{sec:smoothingdynamics}

In \S\ref{sec:smoothings} and \S\ref{sec:Rfssp} we study the Reeb flow on the boundary of a smoothing of a symplectic polytope in $\R^4$. In \S\ref{sec:nss} and \S\ref{sec:srn} we explain some more technical issues arising from the fact that the smoothing is only $C^1$, and in particular how to make sense of the ``rotation number'' of Reeb trajectories. In \S\ref{sec:rnlb} we derive important lower bounds on this rotation number. 

\subsection{Smoothings of polytopes}
\label{sec:smoothings}

If $X\subset\R^m$ is a compact convex set and $\epsilon>0$, define the $\epsilon$-smoothing $X_\epsilon$ of $X$ by equation \eqref{eqn:deltasmoothing}.
Observe that $X_\epsilon$ is convex. Denote its boundary by $Y_\epsilon = \partial X_\epsilon$. We now describe $Y_\epsilon$ more explicitly, in a way which mostly does not depend on $\epsilon$. We first have:

\begin{lemma}
\label{lem:Yepsilon}
If $X$ is a compact convex set then
\[
Y_\epsilon = \{y \in \R^m\mid \op{dist}(y,X) = \epsilon\}.
\]
\end{lemma}

\begin{proof}
The left hand side is contained in the right hand side because distance to $X$ is a continuous function on $\R^m$. The reverse inclusion holds because given $y\in \R^m$ with $\op{dist}(y,X)=\epsilon$, since $X$ is compact and convex, there is a unique point $x\in X$ which is closest to $y$. By convexity again, $X$ is contained in the closed half-space $\{z\in \R^m\mid \langle z, y-x\rangle \le 0\}$. It follows that $\op{dist}(t(y-x),X)=\epsilon t$ for $t>0$, so that $y\in\partial X_\epsilon$.
\end{proof}

\begin{definition}
If $X\subset\R^m$ is a compact convex set, define the ``blown-up boundary''
\[
Y_0 = \left\{(y,v) \;\big|\; y\in \partial X,\; v\in N_y^+X,\; |v|=1\right\}\subset\partial X \times S^{m-1}.
\]
\end{definition}

We then have the following lemma, which is proved by similar arguments to Lemma~\ref{lem:Yepsilon}:

\begin{lemma}
\label{lem:bub}
Let $X\subset\R^m$ be a compact convex set and let $\epsilon>0$. Then:
\begin{itemize}
\item[\emph{(a)}]
There is a homeomorphism
\[
Y_0\stackrel{\simeq}{\longrightarrow} Y_\epsilon
\]
sending $(y,v)\mapsto y+\epsilon v$.
\item[\emph{(b)}]
The inverse homeomorphism sends $y\mapsto (x,\epsilon^{-1}(y-x))$ where $x$ is the unique closest point in $X$ to $y$.
\item[\emph{(c)}]
For $y\in Y_\epsilon$, if $x$ is the closest point in $X$ to $y$, then the positive normal cone $N_y^+X_\epsilon$ is the ray consisting of nonnegative multiples of $y-x$.
\end{itemize}\end{lemma}

Suppose now that $X\subset\R^m$ is a convex polytope and $\epsilon>0$.

\begin{definition}
If $F$ is a face of $X$, define the {\bf $\epsilon$-smoothed face\/}
\[
F_\epsilon = \{x \in Y_\epsilon \mid \op{dist}(x,F)=\epsilon\}.
\]
\end{definition}

By Lemma~\ref{lem:bub}, we have
\[
Y_\epsilon = \bigsqcup_F F_\epsilon
\]
and
\[
F_\epsilon = F + \{v\in N_F^+X\mid |v|=\epsilon\}.
\]
Note that each $F_\epsilon$ is a $C^\infty$ smooth hypersurface, and where the closure of one $F_\epsilon$ meets another, the outward unit normal vectors agree. It follows that $Y_\epsilon$ is a $C^1$ smooth hypersurface, and it is $C^\infty$ except along strata\footnote{We do not also need to mention strata of the form $F+\partial \{v\in N_F^+X\mid |v|=\epsilon\}$, because any point in $\partial N_F^+X$ is contained in $N_E^+X$ where $E$ is a face with $F\subset \partial E$.} of the form $\partial F + \{v\in N_F^+X\mid |v|=\epsilon\}$.

\subsection{The Reeb flow on a smoothed symplectic polytope}
\label{sec:Rfssp}

Suppose now that $X$ is a symplectic polytope in $\R^4$ and $\epsilon>0$. As noted above, $Y_\epsilon = \partial X_\epsilon$ is a $C^1$ convex hypersurface, and as such it has a well-defined $C^0$ Reeb vector field, which is smooth except along the strata of $Y_\epsilon$ arising from the boundaries of the faces of $X$. We now investigate the Reeb flow on $Y_\epsilon$ in more detail, as well as the lifted linearized Reeb flow $\widetilde{\phi}$ from Definition~\ref{def:linearized}.

\subsection*{General remarks.}

By Lemma~\ref{lem:bub}, a point in $Y_\epsilon$ lives in an $\epsilon$-smoothed face $F_\epsilon$ for a unique face $F$ of $X$, and thus has the form $y+\epsilon v$ where $y\in F$ and $v\in N_F^+X$ is a unit vector. By equation \eqref{eqn:Reebinu} and Lemma~\ref{lem:bub}(c), the Reeb vector field at $y+\epsilon v$ is given by
\begin{equation}
\label{eqn:Reebsmoothed}
R_{y+\epsilon v} = \frac{2{\mathbf i}v}{\langle v,y\rangle + \epsilon}.
\end{equation}

\begin{lemma}
\label{lem:transinv}
The Reeb vector field \eqref{eqn:Reebsmoothed} on the $\epsilon$-smoothed face $F_\epsilon$, regarded as a map $F_\epsilon\to \R^4$, depends only $v\in N_F^+X$ and not on the choice of $y\in F$.
\end{lemma}

\begin{proof}
This follows from equation \eqref{eqn:Reebsmoothed}, because for fixed $v\in N_F^+X$ and for two points $y,y'\in F$, by the definition of positive normal cone we have $\langle v,y-y'\rangle = 0$.
\end{proof}

\subsection*{Smoothed 3-faces.} The Reeb flow on a smoothed $3$-face is very simple.

\begin{lemma}
\label{lem:smoothed3face}
Let $X\subset\R^4$ be a symplectic polytope, let $\epsilon>0$, and let $E$ be a $3$-face of $X$ with outward unit normal vector $\nu$.
\begin{itemize}
\item[\emph{(a)}]
The Reeb vector field on $E_\epsilon$, regarded as a map $E_\epsilon\to\R^4$, agrees with the Reeb vector field on $E$, up to rescaling by a positive constant which limits to $1$ as $\epsilon\to0$.
\item[\emph{(b)}]
If $\gamma:[0,t]\to E_\epsilon$ is a Reeb trajectory, then $\widetilde{\phi}(\gamma(0),t)=1\in\widetilde{\op{Sp}}(2)$.
\item[\emph{(c)}]
If $y\in\partial E$, then at the point $y+\epsilon\nu\in Y_\epsilon$, the Reeb vector field on $Y_\epsilon$ is not tangent to $\partial E_\epsilon$.
\end{itemize}
\end{lemma}

\begin{proof}
(a)
This follows from equation \eqref{eqn:Reebsmoothed}.

(b) For $s\in[0,t]$, the Reeb flow $\Phi_s:Y_\epsilon\to Y_\epsilon$ is a translation on a neighborhood of $\gamma(0)$. Consequently the linearized Reeb flow $d\Phi_s:\xi_{\gamma(0)}\to \xi_{\gamma(s)}$ is the identity, if we regard $\xi_{\gamma(0)}$ and $\xi_{\gamma(s)}$ as (identical) two-dimensional subspaces of $\R^4$. The quaternionic trivialization $\tau:\R^2\to \xi_{\gamma(s)}$ likewise does not depend on $s\in[0,t]$. Consequently $\phi(y,s)=1$ for all $s\in[0,t]$. Thus $\widetilde{\phi}(y,t)$ is the constant path at the identity in $\op{Sp}(2)$.

(c)
It is equivalent to show that the Reeb vector field on $E$ at $y$ is not tangent to $\partial E$. If the Reeb vector field on $E$ at $y$ is tangent to $\partial E$, then it is tangent to some $2$-face $F\subset \partial E$. By Lemma~\ref{lem:la}, the face $2$-face $F$ is Lagrangian, contradicting our hypothesis that the polytope $X$ is symplectic.
\end{proof}

\subsection*{Smoothed 2-faces.}
Let $F$ be a $2$-face. Let $E_1$ and $E_2$ be the $3$-faces adjacent to $F$. By Lemma~\ref{lem:Reebcone}, we can choose these so that $R_{E_2}$ points out of $F$; and a similar argument shows that then $R_{E_1}$ points into $F$. Let $\nu_1$ and $\nu_2$ denote the outward unit normal vectors to $E_1$ and $E_2$ respectively. By Lemma~\ref{lem:ncn}, the normal cone $N_F^+$ consists of nonnegative linear combinations of $\nu_1$ and $\nu_2$. Let $\{v,w\}$ be an orthonormal basis for $F^\perp$, such that the orientation given by $(v,w)$ agrees with the orientation given by $(\nu_1,\nu_2)$. For $i=1,2$ we can write $\nu_i=(\cos\theta_i) v + (\sin\theta_i) w$ where $0<\theta_2-\theta_1 < \pi$. We then have a homeomorphism
\begin{equation}
\label{eqn:smoothed2face}
\begin{split}
F \times [\theta_1,\theta_2] &\stackrel{\simeq}{\longrightarrow} F_\epsilon,\\
(y,\theta) &\longmapsto 
y+\epsilon((\cos\theta) v + (\sin\theta) w).
\end{split}
\end{equation}

In the coordinates $(y,\theta)$, the Reeb vector field $R$ on $F_\epsilon$ depends only on $\theta$ by Lemma~\ref{lem:transinv}, and has positive $\partial_\theta$ coordinate for both $\theta=\theta_1$ and $\theta=\theta_2$ by our choice of labeling of $E_1$ and $E_2$. By equation \eqref{eqn:Reebsmoothed}, Lemma~\ref{lem:la}, and our hypothesis that the polytope $X$ is symplectic, the $\partial_\theta$ component of the Reeb vector field is positive on all of $F_\epsilon$.

Let $U_{F,\epsilon}\subset F$ denote the set of $y\in F$ such that the Reeb flow on $Y_\epsilon$ starting at $(y,\theta_1)\in F_\epsilon$ stays in $F_\epsilon$ until reaching a point in $F\times\{\theta_2\}$, which we denote by $(\phi_{F,\epsilon}(y),\theta_2)$. Thus we have a well-defined ``flow map'' $\phi_{F,\epsilon}: U_{F,\epsilon}\to F$.

\begin{lemma}
\label{lem:stf}
Let $F$ be a two-face of a symplectic polytope $X\subset \R^4$. Then:
\begin{itemize}
\item[\emph{(a)}] The flow map $\phi_{F,\epsilon}:U_{F,\epsilon}\to F$ above is translation by a vector $V_{F,\epsilon}\in TF$.
\item[\emph{(b)}]
$|V_{F,\epsilon}|=O(\epsilon)$ and $\lim_{\epsilon\to 0}U_{F,\epsilon}=F$.
\item[\emph{(c)}]
Let $y\in U_{F,\epsilon}$ and let $t$ be the Reeb flow time on $F_\epsilon$ from $y+\epsilon\nu_1$ to $\phi_{F,\epsilon}(y)+\epsilon\nu_2$. Then $\phi(y,t)\in\op{Sp}(2)$ agrees with the transition matrix $\psi_F$ in Definition~\ref{def:transitionmatrix}, and $\widetilde{\phi}(y,t)\in\widetilde{\op{Sp}}(2)$ is the unique lift of $\psi_F$ with rotation number in the interval $(0,1/2)$.
\end{itemize}
\end{lemma}

\begin{proof}
(a) If $y,y'\in U_{F,\epsilon}$, then it follows from the translation invariance in Lemma~\ref{lem:transinv} that $\phi_{F,\epsilon}(y)-y=\phi_{F,\epsilon}(y')-y'$, so $\phi_{F,\epsilon}$ is a translation.

(b) It follows from equation \eqref{eqn:Reebsmoothed} that for each $v$, the Reeb vector field $R_{y+\epsilon v}$, regarded as a vector in $\R^4$, has a well-defined limit as $\epsilon\to 0$, which by Lemma~\ref{lem:la} is not tangent to $F$. Since $\partial_\theta$, regarded as a vector in $\R^4$, has length $\epsilon$, it follows that the flow time of the Reeb vector field on $F_\epsilon$ from $F\times\{\theta_1\}$ to $F\times\{\theta_2\}$ is $O(\epsilon)$. Consequently the translation vector $V_{F,\epsilon}$ has length $O(\epsilon)$, and the complement $F\setminus U_{F,\epsilon}$ of the domain of the flow map is contained within distance $O(\epsilon)$ of $\partial F$.

(c) Write $y_1=y+\epsilon\nu_1$ and $y_2=\phi_{F,\epsilon}(y)+\epsilon\nu_2$. By part (a) and the translation invariance in Lemma~\ref{lem:transinv}, the time $t$ Reeb flow $\Phi_t$ on $Y_\epsilon$ restricted to $U_{F,\epsilon} + \epsilon \nu_1$ is a translation in $\R^4$.
%by $V_{F,\epsilon} + \epsilon(\nu_2-\nu_1)$ to its image in $F+\epsilon\nu_2$. 
Hence the derivative of $\Phi_t$ on the full tangent space of $Y_\epsilon$, namely
\[
d\Phi_t: T_{y_1}Y_\epsilon \longrightarrow T_{y_2}Y_\epsilon,
\]
restricts to the identity on $TF$. We now have a commutative diagram
\[
\begin{CD}
\xi_{y_1} @>>> TF @>{\tau_F'}>> \R^2 \\
@V{d\Phi_t}VV @V{1}VV @VV{\psi_F}V \\
\xi_{y_2} @>>> TF @>{\tau_F}>> \R^2.
\end{CD}
\]
Here the upper left horizontal arrow is projection along the Reeb vector field in $T_{y_1}Y_\epsilon$, and the lower left horizontal arrow is projection along the Reeb vector field in $T_{y_2}Y_\epsilon$. The right horizontal arrows were defined in Definition~\ref{def:qtfg} and Remark~\ref{rem:altcon}. The left square commutes because $d\Phi_t$ preserves the Reeb vector field. The right square commutes by Definition~\ref{def:transitionmatrix}. The composition of the arrows in the top row is the quaternionic trivialization $\tau$ on $\xi_{y_1}$, and the composition of the arrows in the bottom row is the quaternionic trivialization $\tau$ on $\xi_{y_2}$. Going around the outside of the diagram then shows that $\phi(y,t)=\psi_F$.

To determine the lift $\widetilde{\phi}(y,t)$, note that this is actually defined for, and depends continuously on, any $\epsilon>0$ and any pair of hyperplanes $E_1$ and $E_2$ that do not contain the origin and that intersect in a non-Lagrangian $2$-plane $F$. Thus we can denote this lift by $\widetilde{\phi}(E_1,E_2,\epsilon)\in\widetilde{\op{Sp}}(2)$.
Now fixing $E_1$, $F$, and $\epsilon$, we can interpolate from $E_1$ and $E_2$ via a $1$-parameter family of hyperplanes $\{E_s\}_{s\in[1,2]}$ such that $0\notin E_s$ and $E_1\cap E_s=F$ for $1<s\le 2$. The rotation number $\rho:\widetilde{\op{Sp}}(2)\to\R$ then gives us a continuous map
\[
\begin{split}
f: (1,2] &\longrightarrow \R,\\
s &\longmapsto \rho\left(\widetilde{\phi}(E_1,E_s,\epsilon)\right)
\end{split}
\]
We have $\lim_{\tau\searrow 1}\widetilde{\phi}(E_1,E_s,\epsilon)=1$, so $\lim_{s\searrow 1}f(s) = 0$. On the other hand, for each $s\in(1,2]$, the fractional part of $f(s)$ is in the interval $(0,1/2)$ by Lemma~\ref{lem:transitionmatrix}. It follows by continuity that $f(s)\in(0,1/2)$ for all $s\in(1,2]$. Thus $f(2)\in(0,1/2)$, which is what we wanted to prove.
\end{proof}

\subsection*{Smoothed 1-faces.} The Reeb flow on a smoothed $1$-face is more complicated, but we will not need to analyze this in detail. We just remark that one can see the difference between good and bad $1$-faces in the Reeb dynamics on their smoothings. Namely:

\begin{remark}
\label{rem:spiral}
If $L$ is a bad $1$-face, then by definition, there is a unique unit vector $v\in N_L^+X$ such that ${\mathbf i}v$ is tangent to $L$. The line segment $L+\epsilon v\subset L_\epsilon$ is then a Reeb trajectory. On the complement of this line in $L_\epsilon$, the Reeb vector field spirals around the line, with the number of times that it spirals around going to infinity as $\epsilon\to 0$. This gives some intuition why Type 3 combinatorial Reeb orbits do not correspond to limits of sequences of Reeb orbits on smoothings with bounded rotation number.
\end{remark}

By contrast, if $L$ is a good $1$-face, then the Reeb vector field on $L_\epsilon$ always has a nonzero component in the $N_L^+X$ direction.

\subsection*{Smoothed 0-faces.} If $P$ is a $0$-face, then by Lemma~\ref{lem:bub}, $P_\epsilon$ is identified with a domain in $S^3$. By equation \eqref{eqn:Reebsmoothed}, the Reeb vector field on this domain agrees, up to reparametrization, with the standard Reeb vector field on the unit sphere in $\R^4$.

\subsection{Non-smooth strata}
\label{sec:nss}

We now investigate in more detail how Reeb trajectories on $Y_\epsilon$ intersect the strata where $Y_\epsilon$ is not $C^\infty$.

Let $\Sigma$ denote the subset of $Y_\epsilon$ where $Y_\epsilon$ is not locally $C^\infty$. By the discussion at the end of \S\ref{sec:smoothings}, we can write
\[
\Sigma = \Sigma_1 \sqcup \Sigma_2 \sqcup \Sigma_3
\]
where:

\begin{itemize}
\item
$\Sigma_1$ is the disjoint union of sets
\begin{equation}
\label{eqn:Sigma1}
P+\{v\in N_L^+X\mid |v|=\epsilon\}
\end{equation}
where $P$ is a vertex of $X$, and $L$ is a $1$-face adjacent to $P$.
\item
$\Sigma_2$ is the disjoint union of sets
\begin{equation}
\label{eqn:Sigma2}
L+\{v\in N_F^+X\mid |v|=\epsilon\}
\end{equation}
where $L$ is a $1$-face, and $F$ is a $2$-face adjacent to $L$.
\item
$\Sigma_3$ is the disjoint union of sets
\[
F+\epsilon\nu
\]
where $F$ is a $2$-face, and $\nu$ is the outward unit normal vector to one of the two $3$-faces $E$ adjacent to $F$.
\end{itemize}

\begin{lemma}
\label{lem:nsrt}
Let $X\subset\R^4$ be a symplectic polytope, let $\epsilon>0$, and let $\gamma:[a,b]\to Y_\epsilon$ be a Reeb trajectory. Then there exist a nonnegative integer $k$ and real numbers $a\le t_1 < t_2 < \cdots < t_k \le b$ with the following properties:
\begin{itemize}
\item[\emph{(a)}]
$\gamma(t_i)\in\Sigma$ for each $i$.
\item[\emph{(b)}]
For each $i=0,\ldots,k$, one of the following possibilities holds:
\begin{itemize}
\item[\emph{(i)}] $\gamma$ maps $(t_i,t_{i+1})$ to $Y_\epsilon\setminus\Sigma$. (Here we interpret $t_0=a$ and $t_{k+1}=b$.)
\item[\emph{(ii)}] $\gamma$ maps $(t_i,t_{i+1})$ to a Reeb trajectory in a component of $\Sigma_1$. (Each component of $\Sigma_1$ contains at most one Reeb trajectory of positive length.)
\item[\emph{(iii)}] $\gamma$ maps $(t_i,t_{i+1})$ to a Reeb trajectory in a component of $\Sigma_2$. (This can only happen when the corresponding $2$-face $F$ is complex linear, and in this case the component of $\Sigma_2$ is foliated by Reeb trajectories.)
\end{itemize}
\end{itemize}
\end{lemma}

\begin{proof}
We need to show that a Reeb trajectory intersects $\Sigma$ in isolated points, or in Reeb trajectories of the types described in (ii) and (iii).

We have seen in \S\ref{sec:Rfssp} that the Reeb vector field is transverse to all of $\Sigma_3$. Thus the Reeb trajectory $\gamma$ intersects $\Sigma_3$ only in isolated points.

Next let us consider the Reeb vector field on a component of $\Sigma_2$ of the form \eqref{eqn:Sigma2}. As in \S\ref{sec:Rfssp}, let $E_1$ and $E_2$ denote the $3$-faces adjacent to $F$, with outward unit normal vectors $\nu_1$ and $\nu_2$ respectively. The smoothing $F_\epsilon$ is parametrized by \eqref{eqn:smoothed2face}. This parametrization extends by the same formula to a parametrization of $\overline{F_\epsilon}$ by $\overline{F}\times [\theta_1,\theta_2]$. The latter parametrization includes the component \eqref{eqn:Sigma2} of $\Sigma_2$ as the restriction to $L\times [\theta_1,\theta_2]$. By equation \eqref{eqn:Reebsmoothed}, at the point corresponding to $(y,\theta)$ in \eqref{eqn:smoothed2face}, the Reeb vector is given by
\begin{equation}
\label{eqn:RSigma2}
R = \frac{2}{\langle (\cos\theta)v+(\sin\theta)w,y\rangle + \epsilon}{\mathbf i}((\cos\theta) v + (\sin\theta) w).
\end{equation}
This vector is tangent to the component \eqref{eqn:Sigma2} if and only if the orthogonal projection of ${\mathbf i}((\cos\theta)v + (\sin\theta)w)$ to $F$ is parallel to $L$.

If the projections of ${\mathbf i}v$ and ${\mathbf i}w$ to $F$ are not parallel, then this tangency will only happen for isolated values of $\theta$, and since the Reeb vector field on $\overline{F_\epsilon}$ always has a positive $\partial_\theta$ component, a Reeb trajectory will only intersect the component \eqref{eqn:Sigma2} in isolated points.

If on the other hand the projections of ${\mathbf i}v$ and ${\mathbf i}w$ to $F$ are parallel, then there is a nontrivial linear combination of ${\mathbf i}v$ and ${\mathbf i}w$ whose projection to $F$ is zero. This means that there is a nonzero vector $\nu$ perpendicular to $F$ such that ${\mathbf i}\nu$ is also perpendicular to $F$. This means that $F^\perp$ is complex linear, and thus $F$ is also complex linear. Then ${\mathbf i}v$ and ${\mathbf i}w$ are both perpendicular to $F$, so in the parametrization \eqref{eqn:smoothed2face}, the Reeb vector field vector field \eqref{eqn:RSigma2} is a just a positive multiple of $\partial_\theta$.

The conclusion is that a Reeb trajectory will intersect each component \eqref{eqn:Sigma2} of $\Sigma_2$ either in isolated points, or (when $F$ is complex linear) in Reeb trajectories which, in the parametrization \eqref{eqn:smoothed2face}, start on $L\times \{\theta_1\}$ and end on $L\times\{\theta_2\}$, keeping the $L$ component constant.

Finally we consider the Reeb vector field on a component \eqref{eqn:Sigma1} of $\Sigma_1$. The set of vectors $v$ that arise in \eqref{eqn:Sigma1} is a domain $D$ in the intersection of the sphere $|v|=\epsilon$ with the hyperplane $L^\perp$. As we have seen at the end of \S\ref{sec:Rfssp}, the Reeb vector field on $Y_\epsilon$ at a point in \eqref{eqn:Sigma1} agrees, up to scaling, with the standard Reeb vector field on the sphere $|v|=\epsilon$, whose Reeb orbits are Hopf circles. There is a unique Hopf circle $C$ contained entirely in $L^\perp$. All other Hopf circles intersect $L^\perp$ transversely. Thus any Reeb trajectory in $Y_\epsilon$ intersects the component \eqref{eqn:Sigma1} in isolated points and/or the arc corresponding to $C\cap D$, if the latter intersection is nonempty.
\end{proof}

\subsection{Rotation number of Reeb trajectories}
\label{sec:srn}

Suppose $\gamma:[a,b]\to Y_\epsilon$ is a Reeb trajectory. Let $D\subset Y_\epsilon$ be a disk through $\gamma(a)$ tranverse to $\gamma$, and let $D'\subset Y_\epsilon$ be a disk through $\gamma(b)$ transverse to $\gamma$. We can identify $D$ with a neighborhood of $0$ in $\xi_{\gamma(a)}$, and $D'$ with a neighborhood of $0$ in $\xi_{\gamma(b)}$, via orthogonal projection in $\R^4$. If $D$ is small enough, then there is a well-defined map continuous map $\phi:D\to D'$ with $\phi(\gamma(a))=\gamma(b)$, such that for each $x\in D$, there is a unique Reeb trajectory near $\gamma$ starting at $x$ and ending at $\phi(x)$.

\begin{lemma}
\label{lem:plrm}
Let $X$ be a symplectic polytope in $\R^4$, let $\epsilon>0$, and let $\gamma:[a,b]\to Y_\epsilon$ be a Reeb trajectory. Then there is a unique (independent of the choice of $D$ and $D'$) homeomorphism
\[
P_\gamma:\xi_{\gamma(a)} \longrightarrow \xi_{\gamma(b)}
\]
such that:
\begin{itemize}
\item[\emph{(a)}]
\begin{equation}
\label{eqn:uniqueP}
\lim_{x\to 0}\frac{\phi(x)-P_\gamma(x)}{\|x\|}=0.
\end{equation}
\item[\emph{(b)}] $P_\gamma$ is linear along rays, i.e. if $x\in \xi_{\gamma(a)}$ and $c>0$ then $P_\gamma(cx) = cP_\gamma(x)$.
\end{itemize}
This map $P_\gamma$ has the following additional properties:
\begin{itemize}
\item[\emph{(c)}]
If $\gamma$ does not include any arcs as in Lemma~\ref{lem:nsrt}(ii)-(iii), and in particular if $\gamma$ does not intersect any smoothed $0$-face or smoothed $1$-face, then $P_\gamma$ is linear.
\item[\emph{(d)}]
For $t\in(a,b)$ we have the composition property
\[
P_\gamma = P_{\gamma|_{[t,b]}} \circ P_{\gamma|_{[a,t]}}.
\]
\item[\emph{(e)}]
For $t\in [a,b]$, the homeomorphism $\R^2\to\R^2$ given by the composition
\[
\R^2 \stackrel{\tau^{-1}}{\longrightarrow} \xi_{\gamma(a)} \stackrel{P_{\gamma|_{[a,b]}}}{\longrightarrow} \xi_{\gamma(t)} \stackrel{\tau}{\longrightarrow} \R^2
\]
is a continuous, piecewise smooth function of $t$.
\end{itemize}
\end{lemma}

\begin{proof}
Uniqueness of the homeomorphism $P_\gamma$ follows from properties (a) and (b). Independence of the choice of $D$ and $D'$ follows from properties (a) and (b) together with continuity of the Reeb vector field. Assuming existence of the homeomorphism $P_\gamma$, the composition property (d) follows from uniqueness.

We now need to prove existence of the homeomorphism satisfying properties (a), (b), (c), and (e). Let $a\le t_1<t_2<\cdots <t_k\le b$ be the subdivision of the inteveral $[a,b]$ given by Lemma~\ref{lem:nsrt}. For $i=0,\ldots,k$, let $\gamma_i$ denote the restriction of $\gamma$ to $[t_i,t_{i+1}]$, where we interpret $t_0=a$ and $t_k=b$. It is enough to prove existence of a homeomorphism
\[
P_{\gamma_i}: \xi_{\gamma(t_i)} \longrightarrow \xi_{\gamma(t_{i+1})}
\]
with the required properties for each $i$. The desired homeomorphism $P_\gamma$ is then given by the composition $P_k\cdots P_0$.

For case (i) in Lemma~\ref{lem:nsrt}, a homeomorphism $P_{\gamma_i}$ with properties (a), (b), and (e) is given by the usual linearized return map on the smooth hypersurface $Y_\epsilon\setminus\Sigma$ from $t_i+\delta$ to $t_{i+1}-\delta$, in the limit as $\delta\to 0$. Since $P_{\gamma_i}$ is linear, we also obtain property (c).

For case (ii) or (iii) in Lemma~\ref{lem:nsrt}, the existence of $P_{\gamma_i}$ with the desired properties follows from the fact that $\gamma_i$ is on a smooth hypersurface separating two regions of $Y_\epsilon$, on each of which the Reeb vector field is $C^\infty$.
\end{proof}

\begin{remark}
\label{rem:avf}
In case (ii) or (iii) above, the description of the Reeb flow in \S\ref{sec:Rfssp} allows us to write down the map $P_{\gamma_i}$ quite explicitly. Namely, for a suitable trivialization, $P_{\gamma_i}$ is given by the flow for some positive time of a continuous, piecewise smooth vector field $V$ on $\R^2$, which is the derivative of a shear on one half of $\R^2$, and which is the derivative of a rotation or the identity on the other half of $\R^2$. For case (ii), the vector field has the form
\begin{equation}
\label{eqn:avf2}
V(x,y) = \left\{\begin{array}{cl} -y\partial_x, & x\ge 0,\\ x\partial_y - y\partial_x, & x\le 0. \end{array}
\right.
\end{equation}
For case (iii), the vector field has the form
\begin{equation}
\label{eqn:avf3}
V(x,y) = \left\{\begin{array}{cl} x\partial_y, & x\ge 0,\\ 0, & x\le 0. \end{array}\right.
\end{equation}
\end{remark}

Since the map $P_\gamma:\xi_{\gamma(a)}\to\xi_{\gamma(b)}$ sends rays to rays, it induces a well-defined map ${\mathbb P}\xi_{\gamma(a)}\to{\mathbb P}\xi_{\gamma(b)}$. It follows from Lemma~\ref{lem:plrm}(c),(d) and equations \eqref{eqn:avf2} and \eqref{eqn:avf3} that the latter map is $C^1$. Similarly to \eqref{eqn:lrf}, we obtain a $C^1$ diffeomorphism of $S^1$ given by the composition
\[
S^1\stackrel{\tau^{-1}}{\longrightarrow} {\mathbb P}\xi_{\gamma(a)} \stackrel{P_\gamma}{\longrightarrow} {\mathbb P}\xi_{\gamma(b)} \stackrel{\tau}{\longrightarrow} S^1. 
\]
Stealing the notation from Definition~\ref{def:linearized}, let us denote this map by $\phi(y,t)$ where $y=\gamma(a)$ and $t=b-a$. By analogy with \eqref{eqn:llrf}, we define
\[
\widetilde{\phi}(y,t) = \{\phi(y,s)\}_{s\in[0,t]}\in\widetilde{\op{Diff}}(S^1).
\]
This then has a well-defined rotation number, see Appendix A, which we denote by
\[
\rho(\gamma) = \rho(\widetilde{\phi}(y,t))\in\R.
\]

\subsection{Lower bounds on the rotation number}
\label{sec:rnlb}

We now prove the following lower bound on the rotation number.

\begin{lemma}
\label{lem:rnlb1}
Let $X$ be a symplectic polytope in $\R^4$. Then there exists a constant $C>0$, depending only on $X$, such that if $\epsilon>0$ is small, then the following holds. Let $\gamma:[a,b]\to Y_\epsilon$ be a Reeb trajectory, and assume that if $t\in(a,b)$ and $E$ is a $3$-face then $\gamma(t)\notin E_\epsilon$. Then
\[
\rho(\gamma)\ge C\epsilon^{-1}(b-a).
\]
\end{lemma}

\begin{proof}
Define a function
\[
r^{\min}_\epsilon:Y_\epsilon\longrightarrow\R
\]
as follows. A point $Y_\epsilon$ can by uniquely written as $y+\epsilon v$ where $y\in Y$ and $v$ is a unit vector in $N_y^+X$. Then define
\begin{equation}
\label{eqn:remin}
r^{\min}_\epsilon(y+\epsilon v) = \min_{\theta\in\R/2\pi\Z}\frac{1}{\pi(\langle v,y\rangle + \epsilon)}(S({\mathbf i}v) + S(\cos(\theta){\mathbf j}v + \sin(\theta){\mathbf k}v)).
\end{equation}
Here $S:TY_\epsilon\to\R$ is the single-argument version of the second fundamental form, which is well-defined, even though along the non-smooth strata of $Y_\epsilon$ there is no corresponding bilinear form.

More explicitly, $T_{y+\epsilon v}Y_\epsilon$, regarded as a subspace of $\R^4$, does not depend on $\epsilon$. A tangent vector $V\in T_{y+\epsilon v}Y_\epsilon$ can be uniquely decomposed as
\begin{equation}
\label{eqn:vtn}
V = V_T + V_N
\end{equation}
where $V_T\in T_y\partial X$ is tangent to a face $F$ such that $y\in\overline{F}$ and $v\in N_F^+X$, and $V_N\in T_vN_y^+X$ is perpendicular to $v$. We then have
\begin{equation}
\label{eqn:svepsilon}
S(V) = \epsilon^{-1}|V_N|^2.
\end{equation}

Lemma~\ref{lem:minrot} and Proposition~\ref{prop:uj} carry over to the present situation to show that
\begin{equation}
\label{eqn:cops}
\rho(\gamma) \ge \int_a^b r_\epsilon^{\min}(\gamma(s))ds.
\end{equation}
In \eqref{eqn:remin}, by compactness, there is a uniform upper bound on $\langle v,y\rangle$ for $y\in\partial X$ and $v\in N_y^+X$ a unit vector. Thus by \eqref{eqn:svepsilon} and \eqref{eqn:cops}, to complete the proof of the lemma, it is enough to show that there is a constant $C>0$ such that
\begin{equation}
\label{eqn:annest}
\left|({\mathbf i}v)_N\right|^2 + \left|(\cos(\theta){\mathbf j}v + \sin(\theta){\mathbf k}v)_N\right|^2 \ge C
\end{equation}
whenever $y\in\partial X$, $v\in N_y^+X$ is a unit vector, $\theta\in\R/2\pi\Z$, and $y+\epsilon v$ is not in the closure of $E_\epsilon$ where $E$ is a $3$-face. To prove this, it is enough to show that for each $k$-face $F$ with $k<3$, there is a uniform positive lower bound on the left hand side of \eqref{eqn:annest} for all $y\in F$, all unit vectors $v$ in $N_F^+X$ that are not normal to a $3$-face adjacent to $F$, and all $\theta$.

If $k=2$, then we have a positive lower bound on $|({\mathbf i}v)_N|^2$ by the discussion of smoothed $2$-faces in \S\ref{sec:Rfssp}.

If $k=1$, denote the $1$-face $F$ by $L$. If $v$ is on the boundary of $N_L^+X$, then we have a positive lower bound on $|({\mathbf i}v)_N|^2$ as in the case $k=2$ above. Suppose now that $v$ is in the interior of $N_L^+X$. We have a positive lower bound on $|({\mathbf i}v)_N|^2$ when ${\mathbf i}v_N$ is away from the Reeb cone of $L$. This is sufficient when $L$ is a good $1$-face. If $L$ is a bad $1$-face, then we have to consider the case where ${\mathbf i}v$ is on or near the Reeb cone $R_L^+X$. If ${\mathbf i}v$ is in the Reeb cone, then all vectors in $V\in T_{y+\epsilon v}Y_\epsilon$ that are not in the real span of the Reeb cone $R_L^+X$ have $V_N\neq 0$. Since the vectors $\cos(\theta){\mathbf j}v + \sin(\theta){\mathbf k}v$ are all unit length and orthogonal to ${\mathbf i}v$, we get a positive lower bound on $\left|(\cos(\theta){\mathbf j}v + \sin(\theta){\mathbf k}v)_N\right|^2$ for all $\theta$ when ${\mathbf i}v$ is on or near the Reeb cone.

Suppose now that $k=0$. If $v$ is on the boundary of $N_L^+X$, then the desired lower bound follows as in the cases $k=1$ and $k=2$ above. If $v$ is in the interior of $N_F^+X$, then we have $|({\mathbf i}v)_N|^2=1$.
\end{proof}

We now deduce a related rotation number bound. Let $\gamma:[a,b]\to Y_\epsilon$ be a Reeb trajectory. By Lemma~\ref{lem:bub}, we can write
\[
\gamma(t) = y(t) + \epsilon v(t)
\]
where $y(t)\in \partial X$ and $v(t)$ is a unit vector in $N_{y(t)}^+X$ for each $t$. 

\begin{lemma}
\label{lem:rnlb2}
Let $X$ be a symplectic polytope in $\R^4$. Then there exists a constant $C>0$, depending only on $X$, such that if $\epsilon>0$ is small and $\gamma:[a,b]\to Y_\epsilon$ is a Reeb trajectory as above, then
\[
\rho(\gamma) \ge C \int_a^b|v'(s)|ds.
\]
\end{lemma}

\begin{proof}
By Lemma~\ref{lem:rnlb1}, it is enough to show that there is a constant $C$ such that
\[
|v'(s)|\le C\epsilon^{-1}.
\]
To prove this last statement, observe that by equation \eqref{eqn:Reebsmoothed}, in the notation \eqref{eqn:vtn} we have
\[
v'(s) = \frac{2\epsilon^{-1}}{\langle v(s),y(s)\rangle + \epsilon}({\mathbf i}v(s))_N.
\]
Thus
\[
|v'(s)|
\le \frac{2\epsilon^{-1}}{\langle v(s),y(s)\rangle + \epsilon}.
\]
If $y\in\partial X$ and $v\in N_y^+X$ is a unit vector, then $\langle v,y\rangle >0$ because $X$ is convex and $0\in\op{int}(X)$. By compactness, there is then a uniform lower bound on $\langle v,y\rangle$ for all such pairs $(y,v)$.
\end{proof}

%% file: s6_smooth_combinatorial.tex
\section{The smooth-combinatorial correspondence}
\label{sec:correspondence}

We now prove Theorems~\ref{thm:combtosmooth} and \ref{thm:smoothtocomb}.

\subsection{From combinatorial to smooth Reeb orbits}
\label{sec:combtosmooth}

We first prove Theorem~\ref{thm:combtosmooth}. In fact we will prove a slightly more precise statement in Lemma~\ref{lem:combtosmooth} below.

 Let $X$ be a symplectic polytope in $\R^4$ and let $\gamma=(L_1,\ldots,L_k)$ be a Type 1 combinatorial Reeb orbit. This means that there are $3$-faces $E_1,\ldots,E_k$ and $2$-faces $F_1,\ldots,F_k$ such that $F_i$ is adjacent to $E_{i-1}$ and $E_i$, and $L_i$ is an oriented line segment in $E_i$ from a point in $F_i$ to a point in $F_{i+1}$ which is parallel to the Reeb vector field on $E_i$. Here the subscripts $i-1$ and $i+1$ are understood to be mod $k$. Below we will regard $\gamma$ as a piecewise smooth parametrized loop $\gamma:\R/T\Z\to X$, where $T=\mc{A}_{\op{comb}}(\gamma)$, which traverses the successive line segments $L_i$ as Reeb trajectories.

\begin{lemma}
\label{lem:combtosmooth}
Let $X$ be a symplectic polytope in $\R^4$, and let $\gamma=(L_1,\ldots,L_k)$ be a nondegenerate Type 1 combinatorial Reeb orbit. Then there exists $\delta>0$ such that for all $\epsilon>0$ sufficiently small:
\begin{itemize}
\item[\emph{(a)}]
There is a unique Reeb orbit $\gamma_\epsilon$ on the smoothed boundary $Y_\epsilon$ such that
\[
|\gamma_\epsilon - \gamma|_{C^0} < \delta.
\]
\item[\emph{(b)}]
$\gamma_\epsilon$ converges in $C^0$ to $\gamma$ as $\epsilon\to 0$.
\item[\emph{(c)}]
$\gamma_\epsilon$ does not intersect $F_\epsilon$ where $F$ is a $0$-face or $1$-face.
\item[\emph{(d)}]
$\gamma_\epsilon$ is linearizable, i.e. has a well-defined linearized return map.
\item[\emph{(e)}]
$\mc{A}(\gamma_\epsilon) - \mc{A}_{\op{comb}}(\gamma) = O(\epsilon)$.
\item[\emph{(f)}]
$\gamma_\epsilon$ is nondegenerate, $\rho(\gamma_\epsilon)=\rho_{\op{comb}}(\gamma)$, and $\op{CZ}(\gamma_\epsilon)=\op{CZ}_{\op{comb}}(\gamma)$.
\end{itemize}
\end{lemma}

\begin{proof}
{\em Setup.\/} For $i=1,\ldots,k$, let $p_i$ denote the initial point of the segment $L_i$. Using the notation $E_i$, $F_i$ above, let $D_i$ denote the set of points $y\in F_i$ such that Reeb flow along $E_i$ starting at $y$ reaches a point in $F_{i+1}$, which we denote by $\phi_i(y)$. Thus we have a well-defined affine linear map
\[
\phi_i:D_i\longrightarrow F_{i+1}.
\]
and by definition $\phi_i(p_i) = p_{i+1}$. In particular, the composition
\[
\phi_k\circ\cdots\circ \phi_1: F_1 \longrightarrow F_1
\]
is an affine linear map defined in a neighborhood of $p_1$ sending $p_1$ to itself. For $V\in TF_1$ small, this composition sends
\[
p_1 + V \longmapsto p_1 + AV,
\]
where $A$ is a linear map $TF_1\to TF_1$. Since the combinatorial Reeb orbit $\gamma$ is assumed nondegenerate, the linear map $A$ does not have $1$ as an eigenvalue.

By Lemma~\ref{lem:stf}(a), the Reeb flow along the smoothed $2$-face $(F_i)_{\epsilon}$ induces a well-defined map
\begin{equation}
\label{eqn:stf}
\phi_{F_i,\epsilon}: U_{F_i,\epsilon} \longrightarrow F_i
\end{equation}
which is translation by a vector $V_{F_i,\epsilon}$.

{\em Proof of (a).\/} If $\epsilon>0$ is sufficiently small, then $p_i$ is in the domain $U_{F_i,\epsilon}$ for each $i$, and Reeb orbits on $Y_\epsilon$ that are $C^0$ close to $\gamma$ correspond to fixed points of the composition
\begin{equation}
\label{eqn:2kcomp}
\phi_{F_1,\epsilon} \circ \phi_k \circ \cdots \circ \phi_2 \circ \phi_{F_2,\epsilon} \circ \phi_1 : F_1 \longrightarrow F_1.
\end{equation}
It follows from the above that for $V\in TF_1$ small, the composition \eqref{eqn:2kcomp} sends
\begin{equation}
\label{eqn:p1V}
p_1 + V \longmapsto p_1 + AV + W_\epsilon
\end{equation}
where $W_\epsilon\in TF_1$ has length $O(\epsilon)$. Since the linear map $A-1$ is invertible, the affine linear map \eqref{eqn:p1V} has a unique fixed point $p_1+V$ for some $V\in TF_1$. If $\epsilon$ is sufficiently small, this fixed point will also be in the domain of the composition \eqref{eqn:2kcomp}, and thus will correspond to the desired Reeb orbit $\gamma_\epsilon$.

{\em Proof of (b).\/} This holds because for the above fixed point, $V$ has length $O(\epsilon)$.

{\em Proof of (c).\/} The Reeb orbit $\gamma_\epsilon$ does not intersect $F_\epsilon$ where $F$ is a $0$-face or $1$-face, by the definition of the domain of the map \eqref{eqn:stf}.

{\em Proof of (d).\/} This follows from Lemma~\ref{lem:plrm}(c).

{\em Proof of (e).\/} The symplectic action of the Reeb orbit $\gamma_\epsilon$ is the sum of its flow times over the smoothed $2$-faces $(F_i)_\epsilon$, plus the sum of its flow times over the smoothed $3$-faces $(E_i)_\epsilon$. The former sum is $O(\epsilon)$ as explained in the proof of Lemma~\ref{lem:stf}(b). The latter sum is $(1+O(\epsilon))$ times the sum of the corresponding flow times over the $3$-faces $E_i$, and the latter differs from $\mc{A}_{\op{comb}}(\gamma)$ by $O(\epsilon)$, because the fixed point of \eqref{eqn:p1V} has distance $O(\epsilon)$ from $p_1$.

{\em Proof of (f).\/} Let $T_\epsilon$ denote the period of $\gamma_\epsilon$, and let $y_\epsilon$ be a point on the image of $\gamma_\epsilon$ in $E_k$. If $F$ is a $2$-face, let $\widetilde{\psi}_F\in\widetilde{\op{Sp}}(2)$ denote the lift of the transition matrix $\psi_F$ in Definition~\ref{def:transitionmatrix} with rotation number in the interval $(0,1/2)$. By Lemmas~\ref{lem:smoothed3face}(b) and \ref{lem:stf}(c), the lifted return map $\widetilde{\phi}(y_\epsilon,T_\epsilon)$ is given by
\begin{equation}
\label{eqn:liftedreturnmap}
\widetilde{\phi}(y_\epsilon,T_\epsilon) = \widetilde{\psi}_{F_k}\circ \cdots \circ \widetilde{\psi}_{F_1}.
\end{equation}
Nondegeneracy of the combinatorial Reeb orbit $\gamma$ means that the projection
\[
\phi(y_\epsilon,T_\epsilon) = \psi_{F_k}\circ \cdots \circ \psi_{F_1} \in \op{Sp}(2)
\]
does not have $1$ as an eigenvalue, so $\gamma_\epsilon$ is nondegenerate. Moreover, it follows from \eqref{eqn:liftedreturnmap} and the definition of combinatorial rotation number in Definition~\ref{def:crn} that $\rho_{\op{comb}}(\gamma) = \rho(\gamma_\epsilon)$. This implies that $\op{CZ}_{\op{comb}}(\gamma) = \op{CZ}(\gamma_\epsilon)$.
\end{proof}

\subsection{From smooth to combinatorial Reeb orbits}
\label{sec:smoothtocomb}

\begin{proof}[Proof of Theorem~\ref{thm:smoothtocomb}.]
We proceed in four steps.

{\em Step 1.\/}
We claim that for each $i$, the Reeb orbit $\gamma_i$ can be expressed as a concatenation of a finite number, $k_i$, of arcs such that:
\begin{itemize}
\item[(a)]
Each endpoint of an arc maps to the boundary of $E_{\epsilon_i}$ where $E$ is a $3$-face. 
\item[(b)]
For each arc, either:
\begin{itemize}
\item[(i)] There is a $3$-face $E$ such that the interior of the arc maps to $E_{\epsilon_i}$, or
\item[(ii)] No point in the interior of the arc maps to $E_{\epsilon_i}$ where $E$ is a $3$-face.
\end{itemize}
\end{itemize}

The above decomposition follows from parts (a) and (b)(i) of Lemma~\ref{lem:nsrt}, because the boundary of $E_{\epsilon_i}$ where $E$ is a $3$-face is contained in the singular set $\Sigma$. (Note that the decomposition into arcs in Lemma~\ref{lem:nsrt} is a subdivision of the above decomposition into arcs. Moreover, if $k_i>1$, then $k_i$ is even and the arcs alternate between types (i) and (ii).)

{\em Step 2.\/}
We claim now that there is a constant $C>0$, not depending on $i$, such that if $\gamma:[a.b]\to Y_{\epsilon_i}$ is an arc of type (ii) above, then if we write $\gamma(t)=y(t)+\epsilon_iv(t)$ for $y(t)\in\partial X$ and $v(t)\in N_{y(t)}^+X$ a unit vector, then we have
\begin{equation}
\label{eqn:annint}
\int_a^b|v'(s)ds|\ge C.
\end{equation}

To see this, note that by (a) above, there are 3-faces $E$ and $E'$ such that $\gamma(a)\in\overline{E_{\epsilon_i}}$ and $\gamma(b)\in\overline{E'_{\epsilon_i}}$. Then $v(a)=\vu_E$, where $\nu_E$ denotes the outward unit normal vector to $E$, and likewise $v(b)=\nu_{E'}$. If $E\neq E'$, then the integral in \eqref{eqn:annint} is bounded from below by the distance in $S^3$ between $\nu_E$ and $\nu_{E'}$, and this distance has a uniform positive lower bound because $X$ has only finitely many $3$-faces, each with distinct outward unit normal vectors.

We now consider the case where $E=E'$. The proof of Lemma~\ref{lem:rnlb2} shows that there is a neighborhood $U$ of $\nu_E$ in $S^3$, and a constant $C>0$, such that for any point $y+\epsilon_i v\in Y_{\epsilon_i}\setminus E_{\epsilon_i}$ with $v\in U$, with respect to the decomposition \eqref{eqn:vtn}, we have $|({\mathbf i}v)_N|^2\ge C$. By shrinking the the neighborhood $U$, we can replace this last inequalty with $\langle ({\mathbf i}v)_N,\nu_E\rangle > 0$. Since $v'(t)$ is a positive multiple of $({\mathbf i}v(t))_N$, it follows that the path $[a,b]\to S^3$ sending $t\mapsto v(t)$ must initially exit the neighborhood $U$ before returning to $\nu_E$. So in this case, we can take the constant $C$ in \eqref{eqn:annint} to be twice the distance in $S^3$ from $\nu_E$ to $\partial U$.

{\em Step 3.\/} We now show that we can pass to a subsequence so that the sequence of Reeb orbits $\gamma_i$ on $Y_{\epsilon_i}$ converges in $C^0$ to a Type 1 or Type 2 combinatorial Reeb orbit $\gamma$ for $X$.

By Lemma~\ref{lem:rnlb1} and our hypothesis that $\rho(\gamma_i)<R$, we must have $k_i>1$ when $i$ is sufficiently large. Then, by Lemma~\ref{lem:rnlb2} and Step 2, there is an $i$-independent upper bound on $k_i$. We can then pass to a subsequence such that $k_i$ is equal to an even constant $k$.

By compactness, we can pass to a further subsequence such that the endpoints of the $k$ arcs from Step 1 for $\gamma_i$ converge to $k$ points in the $2$-skeleton of $X$. By Lemma~\ref{lem:smoothed3face}, the $k/2$ arcs of type (i) converge to Reeb trajectories on $3$-faces of $X$. On the other hand, by Lemma~\ref{lem:rnlb1}, for each arc of type (ii), the length of its parametrizing interval converges to $0$. A compactness argument also shows that there is an upper bound on the length of the Reeb vector field on $Y_{\epsilon_i}$. It follows that each arc of type (ii) is converging in $C^0$ to a point. Then $\gamma_i$ converges in $C^0$ to a Type 1 or Type 2 combinatorial Reeb orbit consisting of the line segments on $3$-faces given by the limits of the $k/2$ arcs of type (i).

{\em Step 4.\/} To complete the proof, we now prove that the subsequence and limiting orbit constructed above satisfy all of the requirements (i)-(v) of the theorem.

We have proved assertions (i) and (iii). Assertion (ii) follows from the proof of Lemma~\ref{lem:combtosmooth}(e). Assertion (iv) follows from the proof of Lemma~\ref{lem:combtosmooth}(d),(f). Assertion (v) follows from Lemma~\ref{lem:rnlb2} and Step 2. (To get explicit constants $C_F$, one only needs to consider the case $E\neq E'$ in Step 2.)
\end{proof}

%% file: a1_rotation_numbers.tex
\section{Rotation numbers}
\label{app:rotation_numbers} 

Let $\Sgt$ denote the universal cover of the group $\Sg$ of $2\times 2$ real symplectic matrices. Let $\op{Diff}(S^1)$ denote the group of orientation-preserving $C^1$ diffeomorphisms\footnote{For the most part we could work more generally with orientation-preserving homeomorphisms.} of $S^1=\R/\Z$, and let $\widetilde{\op{Diff}}(S^1)$ denote its universal cover. In this appendix, we review two invariants of elements of $\Sgt$, and more generally $\widetilde{\op{Diff}}(S^1)$: the rotation number $\rho$ and the ``minimum rotation number'' $r$. The former is a standard notion in dynamics and is a key ingredient in Theorem~\ref{thm:smoothtocomb}; and we use the latter to bound the former.  We also explain how to use rotation numbers to efficiently compute certain products in $\Sgt$, which is needed for our algorithms.

\subsection{Rotation numbers of circle diffeomorphisms} \label{subsubsec:the_dynamical_rotation_number_of_circle_diffeomorphisms}

We can identify the universal cover $\Dft$ with the group of $C^1$ diffeomorphisms $\Phi:\R\to\R$ which are $\Z$-equivariant in the sense that $\Phi(t+1)=\Phi(t)+1$ for all $t\in\R$. Such a diffeomorphism of $\R$ descends to an orientation-preserving diffeomorphism of $S^1$, and this defines the covering map $\Dft\to\Df$.

\begin{definition}
\label{def:dynamical_rotation_number_for_S1}
Given $\sigma\in S^1$, we define the {\bf rotation number with respect to $\sigma$\/}, denoted by
\[
r_\sigma:\Dft \longrightarrow \R,
\]
as follows. Let $\Phi$ be a $\Z$-equivariant diffeomorphism of $\R$ as above. Let $t\in\R$ be a lift of $\sigma\in\R/\Z$. We then define
\begin{equation}
\label{eqn:def_of_dynamical_rotation_wrts}
r_\sigma(\Phi) = \Phi(t) - t.
\end{equation}
\end{definition}

\begin{definition}
Given $\Phi\in\Dft$, we define the {\bf rotation number\/}
\begin{equation}
\label{eqn:defrhophi}
\rho(\Phi) = \lim_{n\to\infty}\frac{r_\sigma(\Phi^n)}{n} \in \R
\end{equation}
where $\sigma\in S^1$. This limit does not depend on the choice of $\sigma$. Equivalently,
\begin{equation}
\label{eqn:defrhophi2}
\rho(\Phi) = \lim_{n\to\infty}\frac{\Phi^n(t)-t}{n}
\end{equation}
where $t\in\R$.
\end{definition}

Note that we have the $\Z$-equivariance property
\begin{equation}
\label{eqn:rhoequivariant}
\rho(\Phi+1)=\rho(\Phi)+1.
\end{equation}

We can bound the rotation number as follows.

\begin{definition}
We define the {\bf minimum rotation number\/} $r:\Dft\to\R$ by
\begin{equation}
\label{eqn:def_of_dynamical_rotation}
r\left(\Phi\right) = \min_{\sigma\in S^1} r_\sigma\left(\Phi\right).
\end{equation}
\end{definition}

Alternatively, if $\Phi\in\widetilde{\op{Diff}}(S^1)$ is presented as a piecewise smooth path $\{\phi_t\}_{t\in[0,1]}$ in $\op{Diff}(S^1)$ with $\phi_0=\op{id}_{S^1}$, then
\[
r(\Phi) = \min_{\sigma\in S^1}\int_0^1\frac{d}{ds}\phi_s(\sigma)ds.
\]
In particular, it follows that
\begin{equation}
\label{eqn:rminbound}
r(\Phi) \ge \int_0^1\min_{\sigma\in S^1}\left(\frac{d}{ds}\phi_s(\sigma)\right)\,ds.
\end{equation}

It follows from the definitions that
\begin{equation}
\label{eqn:rhor}
\rho(\Phi) \ge r(\Phi).
\end{equation}

\subsection{A partial order}

\begin{definition}
\label{def:order_on_DiffS1}
We define a partial order $\ge$ on $\Dft$ as follows:
\begin{equation}
\Phi \ge \Psi \text{ if and only if } r_s(\Phi) \ge r_s(\Psi) \text{ for all }s \in S^1.
\end{equation}
Equivalently, $\Phi(t)\ge \Psi(t)$ for all $t\in\R$.
\end{definition}

\begin{lemma}
\label{lem:order_on_DiffS1_invariance}
The partial order $\ge$ on $\Dft$ is left and right invariant.
\end{lemma}

\begin{proof}
Let $\Phi,\Psi,\Theta \in \Dft$, and suppose that $\Phi\ge\Psi$, i.e.
\begin{equation}
\label{eqn:phigreaterthanpsi}
\Phi(t) \ge \Psi(t)
\end{equation}
for every $t\in\R$. We need to show that $\Phi\Theta\ge \Psi\Theta$ and $\Theta\Phi\ge\Theta\Psi$.

Since $\Theta:\R\to\R$ is an orientation preserving diffeomorphism, it preserves the order on $\R$, so it follows from \eqref{eqn:phigreaterthanpsi} that
\[
\Theta(\Phi(t)) \ge \Theta(\Psi(t))
\]
for every $t\in\R$, so $\Theta\Phi\ge\Theta\Psi$.

On the other hand, replacing $t$ by $\Theta(t)$ in the inequality \eqref{eqn:phigreaterthanpsi}, we deduce that
\[
\Phi(\Theta(t)) \ge \Psi(\Theta(t))
\]
for every $t\in\R$, so $\Phi\Theta\ge\Psi\Theta$.
\end{proof}

\begin{lemma}
\label{lem:rhoorder}
If $\Phi,\Psi\in\Dft$ and $\Phi\ge\Psi$, then $\rho(\Phi)\ge \rho(\Psi)$.
\end{lemma}

\begin{proof}
By \eqref{eqn:defrhophi2}, it is enough to show that given $t\in\R$, we have $\Phi^n(t)\ge \Psi^n(t)$ for each positive integer $n$. This follows by induction on $n$, using the fact that $\Phi$ preserves the order on $\R$.
\end{proof}

\subsection{Rotation numbers of symplectic matrices}
\label{subsubsec:the_symplectic_rotation_number}

There is a natural homomorphism $\Sg\to\Df$, sending a symplectic linear map $A:\R^2\to\R^2$ to its action on the set of positive rays (identified with $\R/\Z$ by the map sending $t\in\R/\Z$ to the ray through $e^{2\pi i t}$). This lifts to a canonical homomorphism $\Sgt\to\Dft$. Under this homomorphism, the invariants $r_s$, $r$, and $\rho$ defined above pull back to functions $\Sgt\to\R$, which by abuse of notation we denote using the same symbols.

We can describe the rotation number $\rho:\Sgt\to\R$ more explicitly in terms of the following classification of elements of the symplectic group $\Sg$.

\begin{definition}
\label{def:classifySp2}
Let $A \in \Sg$. We say that $A$ is
\begin{itemize}
	\item {\bf positive hyperbolic} if $\Tr(A) > 2$ and {\bf negative hyperbolic} if $\Tr(A) < -2$.
	\item a {\bf positive shear} if $\Tr(A) = 2$ and a {\bf negative shear} if $\Tr(A) = -2$.
	\item {\bf positive elliptic} if $-2 < \Tr(A) < 2$ and $\det([v,Av]) > 0$ for all $v \in \R^2\setminus\{0\}$.
	\item {\bf negative elliptic} if $-2 < \Tr(A) < 2$ and $\det([v,Av]) < 0$ for all $v \in \R^2\setminus\{0\}$.
\end{itemize}
\end{definition}

By the equivariance property \eqref{eqn:rhoequivariant}, the rotation number $\rho:\Sgt\to\R$ descends to a ``mod $\Z$ rotation number'' $\bar{\rho}:\Sg\to\R/\Z$.

\begin{lemma}
\label{lem:compute_rho_bar}
The mod $\mathbb{Z}$ rotation number $\bar{\rho}:\Sg \to \R/\Z$ can be computed as follows:
\[
\bar{\rho}(A) = \left\{
\begin{array}{ccc}
0 & \text{ if } & A \text{ is positive hyperbolic or a positive shear,}\\
\frac{1}{2} & \text{ if } & A \text{ is negative hyperbolic or a negative shear,}\\
\theta & \text{ if } & A \text{ is positive elliptic with eigenvalues }e^{\pm 2 \pi i \theta} \text{ for } \theta \in (0,\frac{1}{2}),\\
-\theta & \text{ if } & A \text{ is negative elliptic with eigenvalues }e^{\pm 2 \pi i \theta} \text{ for } \theta \in (0,\frac{1}{2}).\\
\end{array}
\right.
\]
\end{lemma}

\begin{proof}
In the first two cases, $A$ has $1$ or $-1$ as an eigenvalue. This means that there exists $s\in S^1$ which is fixed or sent to its antipode, and one can use this $s$ in the definition \eqref{eqn:defrhophi}.

In the third case, $A$ is conjugate to rotation by $2\pi\theta$. One can then lift $A$ to an element of $\Sgt$ whose image in $\Dft$ is a $\Z$-equivariant diffeomorphism $\Phi:\R\to\R$ such that $|\Phi^n(t)-t-n\theta|<1$ for each $t\in\R$. It then follows from \eqref{eqn:defrhophi2} that $\rho(\Phi)=\theta$. The last case is analogous.
\end{proof}

\subsection{Computing products in $\Sgt$}
\label{subsec:computing_with_Sp2}

Observe that $\Sgt$ can be identified with the set of pairs $(A,r)$, where $A\in\Sg$ and $r\in\R$ is a lift of $\overline{\rho}(A)\in\R/\Z$. The identification sends a lift $\widetilde{A}$ to the pair $(A,\rho(\widetilde{A}))$.

For computational purposes, we can keep track of the lifts of $A$ using less information, which is useful when for example we do not want to compute $\overline{\rho}(A)$ exactly. Namely, we can identify a lift $\widetilde{A}$ with a pair $(A,r)$, where $r$ is either an integer (when $A$ has positive eigenvalues), an open interval $(n,n+1/2)$ for some integer $n$ (when $A$ is positive elliptic), a half-integer (when $A$ has negative eigenvalues), or an open interval $(n-1/2,n)$ (when $A$ is negative elliptic).

The following proposition allows us to compute products in the group $\Sgt$ in terms of the above data, in the cases that we need (see Remark~\ref{rem:ucmult}).

\begin{proposition}
\label{prop:ucmult}
Let $\widetilde{A},\widetilde{B} \in \Sgt$. Suppose that $\rho(\widetilde{A})\in(0,1/2)$. Then
\[
\rho(\widetilde{B}) \le \rho(\widetilde{A}\widetilde{B}) \le \rho(\widetilde{B}) + \frac{1}{2}.
\]
\end{proposition}

To apply this proposition, if for example $\widetilde{B}$ is described by the pair $(B,(m,m+1/2))$, then it follows that $\widetilde{A}\widetilde{B}$ is described by either $(AB,(m,m+1/2))$, $(AB,m+1/2)$, or $(AB,(m+1/2,m+1))$. To decide which of these three possibilities holds, by Lemma~\ref{lem:compute_rho_bar} it is enough to check whether $AB$ is positive elliptic, has negative eigenvalues, or is negative elliptic.

\begin{proof}[Proof of Proposition~\ref{prop:ucmult}.]
Let $\Phi$ and $\Psi$ denote the elements of $\Dft$ determined by $\widetilde{A}$ and $\widetilde{B}$ respectively.
Let $\Theta:\R\to \R$ denote translation by $1/2$. By Lemma~\ref{lem:compute_rho_bar}, $\widetilde{A}$ projects to a positive elliptic element of $\Sg$. It follows that with respect to the partial order on $\Dft$, we have
\[
\op{id}_\R \le \Phi \le \Theta.
\]
By Lemma~\ref{lem:order_on_DiffS1_invariance}, we can multiply on the right by $\Psi$ to obtain
\[
\Psi \le \Phi\Psi \le \Theta\Psi.
\]
Using Lemma~\ref{lem:rhoorder}, we deduce that
\[
\rho(\Psi) \le \rho(\Phi\Psi) \le \rho(\Theta\Psi).
\]
Since $\Psi$ comes from a linear map, it commutes with $\Theta$, so we have
\[
\rho(\Theta\Psi) = \rho(\Psi) + \frac{1}{2}.
\]
Combining the above two lines completes the proof.
\end{proof}